\documentclass{llncs}
\usepackage[utf8]{inputenc}
\usepackage{amsmath}
\usepackage{amsfonts}
\usepackage{amssymb}

\normalsize
\makeatother

\usepackage{extarrows}
\usepackage{todonotes}
\usepackage{tikz}
\usepackage{pgfplots}
\pgfplotsset{compat=1.5}
\pgfplotsset{width=7cm}

\pgfplotstableread[row sep=crcr]{
2  0.33333\\
3  0.45454\\
4  0.59999\\
5  0.72979\\
6  0.82571\\
7  0.89629\\
8  0.93935\\
9  0.96545\\
10 0.98066\\
}\fEslemgraver

\pgfplotstableread[row sep=crcr]{
2  0.90136\\
3  0.96131\\
4  0.98367\\
5  0.99283\\
6  0.99678\\
7  0.99853\\
8  0.99933\\
9  0.99969\\
10 0.99986\\
}\fEslemgroebner

\pgfplotstableread[row sep=crcr]{
2  0.03900\\		
3  0.03900\\
4  0.04007\\
5  0.04007\\	
6  5.31767\\	
7  9.13374\\	
8  15.9948\\
9  28.4415\\ 
10 51.1984\\
}\fEmixingtimegraver

\pgfplotstableread[row sep=crcr]{
2  9.62956\\
3  25.3443\\
4  60.7348\\
5  138.983\\
6  310.156\\
7  681.909\\
8  1484.82\\
9  3210.84\\
10 6905.87\\
}\fEmixingtimegroebner

\pgfplotstableread[row sep=crcr]{
1 0.454545\\
2 0.687432\\
3 0.821473\\
4 0.88219\\
}\sEslemgraver

\pgfplotstableread[row sep=crcr]{
1 0.961312\\
2 0.975799\\
3 0.982007\\
4 0.986166\\
}\sEslemgroebner

\pgfplotstableread[row sep=crcr]{
1 1.2683\\
2 2.66814\\
3 5.08503\\
4 7.99778\\
}\sEmixingtimegraver

\pgfplotstableread[row sep=crcr]{
1 25.3443\\
2 40.8186\\
3 55.0752\\
4 71.7865\\
}\sEmixingtimegroebner

\normalsize
\makeatother

\numberwithin{equation}{section}
\numberwithin{proposition}{section}
\numberwithin{corollary}{section}
\numberwithin{lemma}{section}
\numberwithin{example}{section}

\makeatletter
\let\bfseries=\undefined
\DeclareRobustCommand\bfseries
{\not@math@alphabet\bfseries\mathbf
  \boldmath\fontseries\bfdefault\selectfont}
\makeatother

\newcommand{\fibervert}[3]{\vev_{\veb}\left(#1,#2,#3\right)}

\newcommand{\Z}{{\mathbb Z}}
\newcommand{\Zge}{\Z_{\ge 0}}
\newcommand{\N}{{\mathbb N}}

\newcommand{\F}{\mathcal{F}}

\newcommand{\slackedbox}[1]{\mathfrak{B}^{sl}_{#1}}
\newcommand{\slacked}[2]{\mathrm{SL}\left(#1,#2\right)}
\newcommand{\slackedmoves}[1]{\mathrm{SL}\left(#1\right)}
\newcommand{\normalbox}[1]{\mathfrak{B}_{#1}}
\newcommand{\standardbasis}[1]{\mathrm{E}_{#1}}

\newcommand{\lex}{\prec_{\textsc{lex}}}

\newcommand{\gr}[2]{\ve{\mathfrak g}^k\left(#1,#2\right)}
\newcommand{\UGB}[1]{\mathcal{U}\left(#1\right)}
\newcommand{\RGB}[1]{\mathcal{R}_{\prec}\left(#1\right)}
\newcommand{\GB}{\mathcal{R}_{\lex}{\left(A_k\right)}}

\newcommand{\cbox}[2]{\mathcal{B}_{#2}\left(#1\right)}

\newcommand{\fiber}[2]{\mathcal{F}_{#1,#2}}
\newcommand{\fibergraph}[3]{\textnormal{G}\left(\fiber{#1}{#2},#3\right)}
\newcommand{\normalgraph}[2]{\textnormal{G}(#1,#2)}

\newcommand{\ikern}[2]{\ker\left(#1\right)\cap\Z^{#2}}
\newcommand{\ideal}[1]{\langle #1 \rangle}

\newcommand{\M}{\mathcal{M}}

\renewcommand{\subset}{\subseteq}

\DeclareMathOperator{\supp}{supp}

\def\Graver{{\mathcal G}}

\usepackage{enumerate}
\def\ve#1{\mathchoice{\mbox{\boldmath$\displaystyle\bf#1$}}
{\mbox{\boldmath$\textstyle\bf#1$}}
{\mbox{\boldmath$\scriptstyle\bf#1$}}
{\mbox{\boldmath$\scriptscriptstyle\bf#1$}}}

\newcommand\veb{{\ve b}}

\newcommand\vece{{\ve e}}

\newcommand\veg{{\ve g}}

\newcommand\vem{{\ve m}}

\newcommand\veu{{\ve u}}
\newcommand\vev{{\ve v}}
\newcommand\vew{{\ve w}}
\newcommand\vex{{\ve x}}
\newcommand\vey{{\ve y}}
\newcommand\vez{{\ve z}}
\newcommand\veN{{\ve N}}

\newcommand\vertex{node}
\newcommand\vertices{nodes}

\newcommand{\eoproof}{\hspace*{\fill} $\square$ \vspace{5pt}}

\usepackage{ifthen}
\newcommand{\DeclareBracket}[3]{
  \newcommand{#1}[2][]{%
  \ifthenelse%
  {\equal{##1}{}}%
  {\left#2##2\right#3}%
  {\csname ##1l\endcsname#2##2\csname ##1r\endcsname#3}}}
\DeclareBracket\set\{\}
\DeclareBracket\floor\lfloor\rfloor
\DeclareBracket\ceil\lceil\rceil
\DeclareBracket\bracket[]
\DeclareBracket\paren()

\newenvironment{psmallmatrix}{\left(\smallmatrix}{\endsmallmatrix\right)}

\newcommand\FourBlockBig[5][\relax]{\begin{pmatrix}#2& #3\\#4&#5 \end{pmatrix}\ifx#1\relax\else^{(#1)}\fi}
\newcommand\FourBlock[5][\relax]{\begin{psmallmatrix}#2& #3\\#4&#5 \end{psmallmatrix}\ifx#1\relax\else{^{(#1)}}\fi}

\newcommand\TwoBlock[3][\relax]{[#2,#3]\ifx#1\relax\else{^{(#1)}}\fi}

\newcommand{\T}{{\intercal}} 

\usepackage{hyperref}
\begin{document}
\pagestyle{headings}  

\title{On the Connectivity of Fiber Graphs}
\author{Raymond Hemmecke\inst{1} \and Tobias
Windisch\inst{2}\thanks{The author was supported by
\emph{TopMath}, a graduate program of the \emph{Elite Network
of Bavaria} and the \emph{TUM Graduate School}.
He further acknowledges support from the
\emph{German National Academic Foundation}.}}

\institute{\email{\href{mailto:hemmecke@tum.de}{hemmecke@tum.de}};
Technische Universit\"at M\"unchen, Germany \and
\email{\href{mailto:windisch@ma.tum.de}{windisch@ma.tum.de}};
Technische Universit\"at M\"unchen, Germany}

\date{\today}

\maketitle

\begin{abstract}
We consider the connectivity of fiber graphs with respect to Gr\"obner basis and Graver basis moves.
First, we present a sequence of fiber graphs using moves from a Gr\"obner basis and prove that their edge-connectivity is lowest possible and can have an arbitrarily large distance from the minimal degree.
We then show that graph-theoretic properties of fiber graphs do not depend on the size of the right-hand side. This provides a counterexample to a conjecture of Engstr\"om on the \vertex-connectivity of fiber graphs. Our main result shows that the edge-connectivity in all fiber graphs of this counterexample is best possible if we use moves from Graver basis instead.
\vspace{0.5cm}\newline
\end{abstract}

\section{Introduction}\label{sec:Introduction}
Many applications in statistics require a deeper analysis of the structure of a \emph{fiber} of an integer matrix $A\in\Z^{d\times n}$ with $\ker(A)\cap\Zge^n=\{\mathbf{0}_n\}$ and a vector $\veb\in\Z^d$ defined as
\begin{equation}
	\fiber{A}{\veb}:=\{\veu\in\Zge^n: A\cdot\veu=\veb\}.
\end{equation}
Very often, one needs to sample elements of the set $\fiber{A}{\veb}$ randomly, for example in hypothesis testing for log-linear models~\cite[Chapter~1]{Drton2008}. The assumption $\ker(A)\cap\Zge^n=\{\mathbf{0}_n\}$ makes $\fiber{A}{\veb}$ finite for all $\veb\in\Z^d$. A random sampling on $\fiber{A}{\veb}$ can be realised by performing a random walk on a \emph{fiber graph} $\fibergraph{A}{\veb}{\mathcal{M}}$ which is defined for a set $\mathcal{M}\subset\ker(A)\cap\Z^d$ as the graph on the nodes $\fiber{A}{\veb}$ in which two nodes $\vev,\veu\in\fiber{A}{\veb}$ are adjacent if either $\vev-\veu\in\mathcal{M}$ or $\veu-\vev\in\mathcal{M}$. 
The set $\mathcal{M}$ can be seen as a set of directions or
\emph{moves} one is allowed to choose from during the random walk.
Since random walks on graphs are essentially the same as Markov chains
whose state space equals the \vertex-set of the graph -- in the
context of this paper $\fiber{A}{\veb}$ -- one can ask whether this
Markov chain converges against a stationary distribution. If $\fibergraph{A}{\veb}{\M}$ is connected and non-bipartite, the Markov chain is irreducible and aperiodic and hence convergence towards a stationary distribution is guaranteed~\cite[Theorem~4.9]{Levin2008}. Thus, the study of the connectedness of fiber graphs is an important question in statistics.

The idea of sampling from fiber graphs goes back to the seminal work \cite{Diaconis1998} of Diaconis and Sturmfels. They formulated the connectedness of fiber graphs equivalently in the language of commutative algebra: a set of moves $\mathcal{M}\subset\ikern{A}{n}\setminus\{\mathbf{0}_n\}$ makes the fiber graphs $\fibergraph{A}{\veb}{\mathcal{M}}$ connected for all $\veb\in\Z^d$ simultaneously if and only if the set of polynomials $\{\vex^{\vem^+}-\vex^{\vem^-}: \vem\in\M\}$ generates the \emph{toric ideal} $$I_A:=\ideal{\vex^{\veu^+}-\vex^{\veu^-}:\veu\in\ker(A)\cap\Z^n}.$$ 
The tools of commutative algebra provide a long list of moves which
generate the toric ideal finitely (see \cite[Chapters~3
and~10]{Loera2013}): every reduced Gr\"obner bases of $A$ with respect
to a term ordering $\prec$ on $\Zge^n$, denoted by $\RGB{A}$, the
universal Gr\"obner basis of $A$, denoted by $\UGB{A}$, and the Graver
basis of $A$, denoted by $\Graver(A)$, are Markov bases of $A$. We
call a fiber graph using moves from a Gr\"obner basis a
\emph{Gr\"obner fiber graph} and a fiber graph using moves from the
Graver basis a \emph{Graver fiber graph}.

When working with Markov chains, it is typical to ask: What can we say about the random walk and the convergence of the corresponding Markov chain? How long do we have to run the random walk until we have a sufficiently good approximation of its stationary distribution? Is there a difference in using moves from $\RGB{A}$ rather then from $\Graver(A)$ (see Example~\ref{exa:FiberGraphs})? In this paper we have a closer look at a more refined structural information of fiber graphs from which we think an answer to these questions can eventually be derived.

  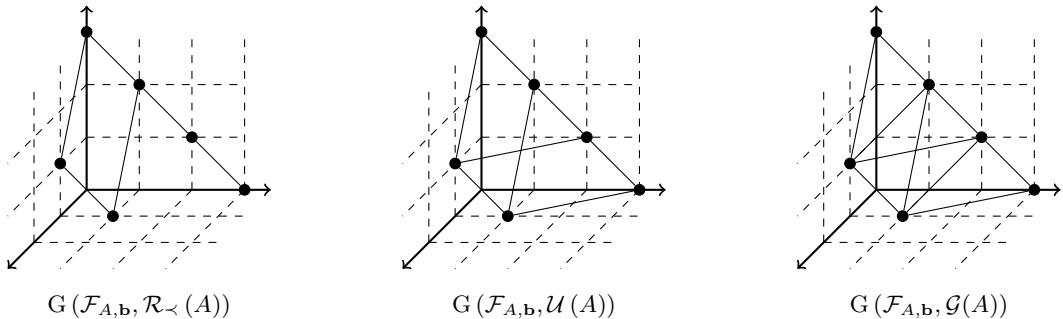
\begin{figure}[htbp]
	\begin{minipage}[b]{0.3\textwidth} 
	\centering
				\begin{tikzpicture}[scale=0.7]
	
	    \draw[->,thick] (0,0)  -- (3.5,0) node (xline) {};
	    \draw[->,thick] (0,0)  -- (0,3.5) node (zline) {};
	    \draw[->,thick] (0,0)  -- (-1.5,-1.5) node (yline) {};

	    \draw[dashed] (0,2)  -- (3,2) node () {};
	    \draw[dashed] (0,1)  -- (3,1) node () {};
	    
	    \draw[dashed] (2,0)  -- (2,3) node () {};
	    \draw[dashed] (1,0)  -- (1,3) node () {};
	    \draw[dashed] (3,0)  -- (3,3) node () {};
	    
	    \draw[dashed] (-0.5,-0.5)  -- (-0.5,2.5) node () {};
	    \draw[dashed] (-1,-1)  -- (-1,2) node () {};
	    
	    \draw[dashed] (0,1)  -- (-1.5,-0.5) node () {};
	    \draw[dashed] (0,2)  -- (-1.5,0.5) node () {};

	    \draw[dashed] (1,0)  -- (-0.5,-1.5) node () {};
	    \draw[dashed] (2,0)  -- (0.5,-1.5) node () {};
	    \draw[dashed] (3,0)  -- (1.5,-1.5) node () {};

	    \draw[dashed] (-0.5,-0.5)  -- (3,-0.5) node (xline) {};
	    \draw[dashed] (-1,-1)  -- (2.5,-1) node (xline) {};

		\node[draw,circle,inner sep=0.05cm,fill=black] (A) at 	(0,3) {}; 
		\node[draw,circle,inner sep=0.05cm,fill=black] (B) at 	(3,0) {}; 
		\node[draw,circle,inner sep=0.05cm,fill=black] (C) at 	(2,1) {}; 
		\node[draw,circle,inner sep=0.05cm,fill=black] (D) at 	(1,2) {}; 
		
		\node[draw,circle,inner sep=0.05cm,fill=black] (E) at 	(0.5,-0.5) {}; 
		\node[draw,circle,inner sep=0.05cm,fill=black] (F) at 	(-0.5,0.5) {}; 
		
		\path(A) edge[font=\small] (D);
		\path(D) edge[font=\small] (C);
		\path(C) edge[font=\small] (B);
		\path(E) edge[font=\small] (F);
		
		\path(A) edge[font=\small] (F);
		\path(D) edge[font=\small] (E);

\end{tikzpicture}
	$\fibergraph{A}{\veb}{\RGB{A}}$
	\end{minipage}\hfill
	\begin{minipage}[b]{0.3\textwidth} 
	\centering
				\begin{tikzpicture}[scale=0.7]
	
	    \draw[->,thick] (0,0)  -- (3.5,0) node (xline) {};
	    \draw[->,thick] (0,0)  -- (0,3.5) node (zline) {};
	    \draw[->,thick] (0,0)  -- (-1.5,-1.5) node (yline) {};

	    \draw[dashed] (0,2)  -- (3,2) node () {};
	    \draw[dashed] (0,1)  -- (3,1) node () {};
	    
	    \draw[dashed] (2,0)  -- (2,3) node () {};
	    \draw[dashed] (1,0)  -- (1,3) node () {};
	    \draw[dashed] (3,0)  -- (3,3) node () {};
	    
	    \draw[dashed] (-0.5,-0.5)  -- (-0.5,2.5) node () {};
	    \draw[dashed] (-1,-1)  -- (-1,2) node () {};
	    
	    \draw[dashed] (0,1)  -- (-1.5,-0.5) node () {};
	    \draw[dashed] (0,2)  -- (-1.5,0.5) node () {};

	    \draw[dashed] (1,0)  -- (-0.5,-1.5) node () {};
	    \draw[dashed] (2,0)  -- (0.5,-1.5) node () {};
	    \draw[dashed] (3,0)  -- (1.5,-1.5) node () {};

	    \draw[dashed] (-0.5,-0.5)  -- (3,-0.5) node (xline) {};
	    \draw[dashed] (-1,-1)  -- (2.5,-1) node (xline) {};

		\node[draw,circle,inner sep=0.05cm,fill=black] (A) at 	(0,3) {}; 
		\node[draw,circle,inner sep=0.05cm,fill=black] (B) at 	(3,0) {}; 
		\node[draw,circle,inner sep=0.05cm,fill=black] (C) at 	(2,1) {}; 
		\node[draw,circle,inner sep=0.05cm,fill=black] (D) at 	(1,2) {}; 
		
		\node[draw,circle,inner sep=0.05cm,fill=black] (E) at 	(0.5,-0.5) {}; 
		\node[draw,circle,inner sep=0.05cm,fill=black] (F) at 	(-0.5,0.5) {};

		\path(A) edge[font=\small] (D);
		\path(D) edge[font=\small] (C);
		\path(C) edge[font=\small] (B);
		\path(E) edge[font=\small] (F);
		
		\path(A) edge[font=\small] (F);
		\path(D) edge[font=\small] (E);
		
		\path(F) edge[font=\small] (C);
		\path(E) edge[font=\small] (B);

\end{tikzpicture}
	$\fibergraph{A}{\veb}{\UGB{A}}$
	\end{minipage}\hfill
		\begin{minipage}[b]{0.3\textwidth} 
	\centering
			\begin{tikzpicture}[scale=0.7]
	
	    \draw[->,thick] (0,0)  -- (3.5,0) node (xline) {};
	    \draw[->,thick] (0,0)  -- (0,3.5) node (zline) {};
	    \draw[->,thick] (0,0)  -- (-1.5,-1.5) node (yline) {};

	    \draw[dashed] (0,2)  -- (3,2) node () {};
	    \draw[dashed] (0,1)  -- (3,1) node () {};
	    
	    \draw[dashed] (2,0)  -- (2,3) node () {};
	    \draw[dashed] (1,0)  -- (1,3) node () {};
	    \draw[dashed] (3,0)  -- (3,3) node () {};
	    
	    \draw[dashed] (-0.5,-0.5)  -- (-0.5,2.5) node () {};
	    \draw[dashed] (-1,-1)  -- (-1,2) node () {};
	    
	    \draw[dashed] (0,1)  -- (-1.5,-0.5) node () {};
	    \draw[dashed] (0,2)  -- (-1.5,0.5) node () {};

	    \draw[dashed] (1,0)  -- (-0.5,-1.5) node () {};
	    \draw[dashed] (2,0)  -- (0.5,-1.5) node () {};
	    \draw[dashed] (3,0)  -- (1.5,-1.5) node () {};

	    \draw[dashed] (-0.5,-0.5)  -- (3,-0.5) node (xline) {};
	    \draw[dashed] (-1,-1)  -- (2.5,-1) node (xline) {};

		\node[draw,circle,inner sep=0.05cm,fill=black] (A) at 	(0,3) {}; 
		\node[draw,circle,inner sep=0.05cm,fill=black] (B) at 	(3,0) {}; 
		\node[draw,circle,inner sep=0.05cm,fill=black] (C) at 	(2,1) {}; 
		\node[draw,circle,inner sep=0.05cm,fill=black] (D) at 	(1,2) {}; 
		
		\node[draw,circle,inner sep=0.05cm,fill=black] (E) at 	(0.5,-0.5) {}; 
		\node[draw,circle,inner sep=0.05cm,fill=black] (F) at 	(-0.5,0.5) {};

		\path(A) edge[font=\small] (D);
		\path(D) edge[font=\small] (C);
		\path(C) edge[font=\small] (B);
		\path(E) edge[font=\small] (F);
		
		\path(A) edge[font=\small] (F);
		\path(D) edge[font=\small] (E);
		
		\path(F) edge[font=\small] (C);
		\path(E) edge[font=\small] (B);

		\path(E) edge[font=\small] (C);
		\path(F) edge[font=\small] (D);

\end{tikzpicture}
	$\fibergraph{A}{\veb}{\Graver(A)}$
	\end{minipage}
	\caption{Different fiber graphs of the same underlying fiber.}\label{fig:Example}
\end{figure}

\begin{example}\label{exa:FiberGraphs}
	Figure~\ref{fig:Example} shows the fiber graphs for the matrix
	$A=(1,1,2)\in\Z^{1\times 3}$ and the right-hand side $\veb=3$
	using different types of moves. Even if we see obvious
	differences in those three fiber graphs, the only statement we
	can make so far is that they are all connected. The mixing
	times of those fiber graphs with respect to the 
	Metropolis-Hastings chain as defined in Section~\ref{sec:ComputationalResults} read from left to right as follows:
	$5.78807$, $6.32917$, and $2.24376$. We see that the mixing
	time of the Graver fiber graph surpasses the mixing time of
	the Gr\"obner fiber graphs by far.
\end{example}

\par{To measure mixing we have to go beyond mere \emph{connectedness}. One possible measurement could be the \emph{connectivity} of the underlying fiber graph (see Section~\ref{sec:Connectivity}) which counts the number of paths between two \vertices.	It can be argued that the connectivity of a graph measures in some sense the possibility of `getting stuck' in a \vertex\ during a random walk and hence a small connectivity cannot lead to a good mixing time of the related Markov chain. In Section~\ref{sec:ComputationalResults} we present our computational results confirming this hypothesis.}

\par{Based on the assumption that a high connectivity is a desirable property of fiber graphs, Engstr\"om conjectured in a talk at IST Austria in 2012 that the \vertex-connectivity is best possible for Gr\"obner fiber graphs.}

\begin{conjecture}[Engstr\"om; 2012]\label{conj:Engstroem:1}
	Let $A\in\Z^{d\times n}$ be a matrix with $\ker(A)\cap\Z^n=\{\mathbf{0}_n\}$ and $\prec$ be a term ordering on $\Z^n$. Then for all $\veb\in\Z^d$, the \vertex-connectivity of $\fibergraph{A}{\veb}{\RGB{A}}$ equals its minimal degree.
\end{conjecture} 

\par{A recent result of Potka supports
Conjecture~\ref{conj:Engstroem:1}. He proved in
\cite{Potka2013} that the \vertex-connectivity of certain
Gr\"obner fiber graphs of the $n\times n$ independence-model
is best possible. However, we show in Section \ref{sec:GroebnerFibers} that Conjecture
\ref{conj:Engstroem:1} is false in general. 
Let $I_k$ be the identity matrix in $\Z^{k\times k}$ and let
$\mathbf{1}_k$ be the $k$-dimensional vector having all entries equal to
$1$. For
\begin{equation}\label{def:Ak}
A_k:=\begin{pmatrix}
I_k & I_k & \mathbf{0} & \mathbf{0} & -\mathbf{1}_k & \mathbf{0} \\
\mathbf{0} & \mathbf{0} & I_k & I_k &  \mathbf{0}  & -\mathbf{1}_k\\
\mathbf{0} & \mathbf{0} & \mathbf{0} & \mathbf{0} & 1 & 1
\end{pmatrix}\in\Z^{(2k+1)\times(4k+2)}
\end{equation}
the underlying fiber graph of $\fiber{A_k}{\vece_{2k+1}}$ has node-connectivity $1$ and minimal degree $k$ when using moves from the reduced lexicographic Gr\"obner basis of $A_k$ (see Corollary \ref{cor:CEconj1}).
Hence, we cannot expect to have a best possible connectivity in
\emph{all} Gr\"obner fiber graphs. Thus,
\cite{Potka2013} poses a weaker follow-up conjecture which claims that
the \vertex-connectivity in fiber graphs is best possible if the
right-hand side is sufficently large.}

\begin{conjecture}[\cite{Potka2013}]\label{conj:Engstroem:2}
Let $A\in\Z^{d\times n}$ be a matrix and $\prec$ be a term ordering. There exists $\veN\in\Zge^d$ such that for all $\veb\in\Zge^d$ with $\veb_i\ge\veN_i$ for all $i\in[d]$, the \vertex-connectivity of the fiber graph $\fibergraph{A}{\veb}{\RGB{A}}$ equals its minimal degree.
\end{conjecture}

\par{We prove in Section~\ref{sec:Connectivity} that once we observe a
bad connectivity in an arbitrary fiber of a matrix, we can construct
fibers of a related matrix whose right-hand side entries exceed any
given bound and whose connectivity remains bad. Thus, by modifying our
original counterexample \eqref{def:Ak}, we show in Section~\ref{sec:GroebnerFibers} that this gives rise to a counterexample to Conjecture~\ref{conj:Engstroem:2}.}

\par{Since these results diminish the hope for suitable connectivity in Gr\"obner fiber graphs, we consider in Section~\ref{sec:GraverFiber} a possible way out. We show that in all Graver fiber graphs of $A_k$, in particular even in those in which the Gr\"obner connectivity is lowest possible, the edge-connectivity best possible.}

\section{Connectivity and Fiber Graphs}\label{sec:Connectivity}

\par{In this section we recall some basic definitions from graph theory and introduce the framework of graph connectivity. We refer to \cite{Diestel2000} for a more general introduction to this field. Let $G=(V,E)$ be a simple graph in finitely many \vertices\ $V$ and edges $E$. In this notation, a fiber graph $\fibergraph{A}{\veb}{\M}$ can be written as $(\fiber{A}{\veb},\{\{\veu,\vev\}: \veu-\vev\in\pm\M\})$. We call
\begin{equation*}
\delta(G):=\min\{ \deg(v): v\in V\}
\end{equation*}
the \emph{minimal degree} of $G$ where $\deg(v)$ is the
cardinality of the neighborhood of $v$ in $G$. Let $k\in\Zge$, then
$G$ is \emph{$k$-\vertex-connected} if $|V|>k$ and if for all $X\subset V$
such that $|X|<k$, the induced graph of $G$ on the
\vertices\ $V\setminus X$ is connected.
In addition, the \emph{\vertex-connectivity} of $G$ is
\begin{equation*}
\kappa(G):=\max\{k\in\Zge: G\textnormal{ is
}k\textnormal{-\vertex-connected}\}.
\end{equation*}
Similarly, $G$ is \emph{$k$-edge-connected} if $|E|>k$ and if for all
$X\subset E$ such that $|X|<k$ the graph $(V,E\setminus X)$ is
connected. The
\emph{edge-connectivity} of $G$ is
\begin{equation*}
\lambda(G):=\max\{k\in\Zge: G\textnormal{ is
}k\textnormal{-edge-connected}\}.
\end{equation*}
For every graph $G$ we have 
$\delta(G)\ge\lambda(G)\ge\kappa(G)$ \cite[Chapter~1.4]{Diestel2000}.
For example, we obtain $\delta(G)\ge\lambda(G)$ by removing all
adjacent edges from a \vertex\ with minimal degree in $G$, which
isolates this \vertex\ and hence gives a disconnected graph. The edge-connectivity of $G$ is \emph{best possible} if $\delta(G)=\lambda(G)$ and similarly the
\vertex-connectivity is \emph{best possible} if
$\delta(G)=\kappa(G)$. Even if
these definitions look very convenient at a first glance, they are
rather unwieldy for proving general results about fiber graphs. For
our purposes, an equivalent property based on the number of paths
between two \vertices\ turns out to be more useful and enables us to
put hands on the connectivity of fiber graphs (see also Menger's
Theorem \cite[Chapter~3.3]{Diestel2000}). 
To obtain a lower bound on the \vertex-connectivity of a graph,
we only have to determine the number of
paths between \vertices\ whose neighborhoods have a non-empty
intersection rather than between all \vertices\ according to Liu's
criterion \cite{Liu1990}.
Since our main result concerns edge-connectivity, we modify Liu's original criterion and obtain a similar statement involving edge-connectivity (see Lemma~\ref{lem:EdgeConnectivity}). A proof of Liu's criterion can be found in \cite{Bjorner2010} and the idea behind the proof of Lemma~\ref{lem:EdgeConnectivity} is similar, which is the reason why we omit its proof here.}
\begin{lemma}\label{lem:EdgeConnectivity}
Let $k\in\Z_{\ge 0}$ and let $G=(V,E)$ be a connected graph with $|E|>k$. If for
all $u,v\in V$ such that $\{u,v\}\in E$ there are at least $k$ edge-disjoint paths from $u$ to $v$ in
$G$, then we have $\lambda(G)\ge k$.
\end{lemma}

\par{Since $\delta(G)\ge\lambda(G)$, replacing $k$ with $\delta(G)$ in
	Lemma \ref{lem:EdgeConnectivity} gives a
	sufficient condition for the
	edge-connectivity to be best possible. The next lemma is very useful in the proof of Proposition~\ref{prop:CombiningGraphs}.}

\begin{lemma}\label{lem:DifferentPaths}
Let $G=(V,E)$ be a graph and $K\subset V$ with $|K|\le\lambda(G)$ and $v\in
V\setminus K$. Then there are $|K|$ edge-disjoint paths in $G$ connecting all
\vertices\ of $K$ with $v$.
\end{lemma}
\begin{proof}
We extend $G$ to a new graph $G^*$ by adding a \vertex\ $u$ and by inserting edges between $u$ and all \vertices\ of $K$. Since $|K|\le\lambda(G)$, $G^*$ must have edge-connectivity at least $|K|$ due to the fact that removal of $|K|-1$ edges does not disconnect $G^*$. The definition of edge-connectivity gives $|K|$ many edge-disjoint paths from $u$ to $v$ in $G^*$. Clearly, those paths connect all \vertices\ of $K$ with $v$ in $G$ as well and by removing $u$ from $G^*$ we obtain edge-disjoint paths from all \vertices\ of $K$ to $v$ in $G$.\eoproof\end{proof}
\vspace{-0.5cm}
\par{The following proposition helps us out in Section~\ref{sec:GraverFiber}.}
\begin{proposition}\label{prop:CombiningGraphs}
Let $G=(V_1\cup V_2,E)$ with $V_1\cap V_2=\emptyset$ and such that the induced subgraphs on $V_1$ and $V_2$ have both an edge-connectivity of at least $n$. Furthermore, assume that every \vertex\ in $V_1$ has at least $m\ge n+2$ neighbors in $V_2$. Let $v_1\in V_1$ and $v_2\in V_2$ be two adjacent \vertices\ such that there are $m$ \vertex-disjoint paths in $G$ connecting $v_1$ and $v_2$ which only use edges whose end-points are in $V_1$ and $V_2$, respectively. Then there are $m+n$ edge-disjoint paths in $G$ connecting $v_1$ and $v_2$.
\end{proposition}
\begin{proof}
By assumption, we obtain for every $j\in[m]$ a path $P_j$ connecting
$v_1$ and $v_2$ which only uses edges between $V_1$ and $V_2$ (the
edge $\{v_1,v_2\}$ is regarded as a path):
\begin{equation}\label{equ:Pi}
v_1\xlongleftrightarrow{P_j}v_2.
\end{equation} 
Moreover, all these paths are pairwise \vertex-disjoint and hence
pairwise edge-disjoint. Denote by $N(v)$ the neighborhood of a
\vertex\ $v$ in $G$. Since, by assumption, the induced subgraph on
$V_1$ is $n$-edge-connected, we have that $v_1$ has at least $n$
neighbors in $V_1$, i.e., $|N(v_1)\cap V_1|\ge n$. Let
$w_1,\ldots,w_n$ be $n$ arbitrary \vertices\ in the neighborhood of
$v_1$ in $V_1$. Again by assumption, we know that $|N(w_i)\cap V_2|\ge
m$ for all $i\in[n]$. Our goal is to construct additional edge-disjoint paths between
$v_1$ and $v_2$. Since the paths $P_j$ are pairwise
\vertex-disjoint, every $w_i$ could have been used by at most one path
$P_j$. Hence, for every $i\in[n]$, up to two edges going from $w_i$
to $V_2$ could have been used by the paths $P_j$. In particular, there are still $m-2$ edges from
$w_i$ into $V_2$ which have not been used by any of the paths $P_j$.
Since we have $m-2\ge n$ by assumption, each $w_i$ has at least $n$
unused neighbors in $V_2$ and thus we can choose for every $i\in [n]$
a \vertex\ $k_i\in N(w_i)\cap V_2$ such that the edge $\{w_i,k_i\}$ is
not used by any of the paths $P_j$ and such that $k_i\neq k_{i'}$ for
all $i\neq i'$. By construction, $\{v_1,w_i,k_i\}$, $i\in [n]$, give
$n$ pairwise \vertex-disjoint paths from $v_1$ to $k_i$. Since the
induced subgraph on $V_2$ is also $n$-edge-connected we can apply Lemma~\ref{lem:DifferentPaths} on the set $\{k_i: i\in[n]\}$ and the node $v_2$ in the induced subgraph on $V_2$ and we obtain for every $i\in[n]$ a path $Q_i$ connecting $k_i$ with $v_2$ such that all of those paths are pairwise edge-disjoint (note that if there is an $i'\in[n]$ such that $k_{i'}=v_2$ we set $Q_{i'}=\emptyset$ and we still can apply the lemma on the smaller set consisting of the $w_i\neq w_{i'}$). 
All in all, we have for all $i\in [n]$ a path
\begin{equation}\label{equ:QiPaths}
v_1\xlongleftrightarrow{\{v_1,w_i\}}w_i\xlongleftrightarrow{\{w_i,k_i\}} k_i \xlongleftrightarrow{Q_i} v_2
\end{equation}
and by construction these paths are pairwise edge-disjoint. Since for all $i\in[n]$ the path from $v_1$ to $w_i$ stays completely in $V_1$, the path from $Q_i$ only uses edges connecting \vertices\ in $V_2$ and since the edges $\{w_i,k_i\}$ had not been used by the paths $P_j$ which on the other hand only uses edges between $V_1$ and $V_2$, the paths given in \eqref{equ:QiPaths} are pairwise edge-disjoint to all paths $P_j$, $j\in[m]$. This gives $m+n$ pairwise edge-disjoint paths in $G$ connecting $v_1$ and $v_2$.
\eoproof\end{proof}

\par{Our next result states that graph-theoretic properties of fiber graphs are independent of the size of the right-hand side.}

\begin{theorem}[Universality Theorem]\label{thm:IndependenceofMarginSize}
Every Gr\"obner fiber graph of a matrix $A\in\Z^{d\times n}$ is isomorphic to a Gr\"obner fiber graph of a matrix $A'\in\Z^{2d\times(n+d)}$ with arbitrarily large entries of its right-hand.
\end{theorem}
\begin{proof} 
Let $\veb\in\Z^d$ be the right-hand side of an arbitrary fiber of $A$
and let $\mathcal{R}$ be a Gr\"obner basis of $A$ with respect to an
arbitrary term ordering $\prec$ on $\Z^n$. Consider the following matrix:
\begin{equation*}
A':=\begin{pmatrix} A & I_d \\ \mathbf{0} &
I_d\end{pmatrix}\in\Z^{2d\times(n+d)}.
\end{equation*}
Clearly, we have $\ikern{A'}{n+d}=\left\{(\vev,\mathbf{0})^T: \vev\in\ikern{A}{n}\right\}$. 
Thus, we obtain that $\mathcal{R}':=\{(\vev,\mathbf{0})^\T: \vev\in\mathcal{R}\}$ is a Gr\"obner basis of $A'$ with respect to an arbitrary extension $\prec'$ of $\prec$ on $\Z^{n+d}$.
We define for every bound $N\in\Zge$ the right-hand side
$\veb'(N):=\begin{pmatrix}\veb+ \tilde{n}\cdot\mathbf{1}_d\\
	\tilde{n}\cdot\mathbf{1}_d\end{pmatrix}\in\Z^{2d}$ where
$\tilde{n}:=\max\{N,N-\veb_i:i\in[d]\}\in\Zge$. It is easy to see that
have the following correlation between fibers of $A$ and $A'$:
$$\fiber{A'}{\veb'(N)}=\left\{\begin{pmatrix}\vev\\\tilde{n}\cdot\mathbf{1}_d\end{pmatrix}:\vev\in\fiber{A}{\veb}\right\}.$$
Thus, for every $N\in\Zge$ the map $\fiber{A}{\veb}\rightarrow\fiber{A'}{\veb'(N)},\vev\mapsto\begin{pmatrix}\vev\\\tilde{n}\cdot\mathbf{1}_d\end{pmatrix}$ gives a bijection from the \vertices\ of $\fiber{A'}{\veb'(N)}$ to the \vertices\ of $\fiber{A}{\veb}$. Even more, from the relation between the Gr\"obner bases $\mathcal{R}$ and $\mathcal{R}'$, this map respects the set of edges and hence it gives rise to a graph isomorphism between the graphs $\fibergraph{A}{\veb'(N)}{\mathcal{R}'}$ and $\fibergraph{A}{\veb}{\mathcal{R}}$ for all $N\in\Zge$.
Since we have $\veb'(N)\ge N\cdot\ve 1_{2d}$ we found a fiber graph
which is isomorphic to $\fibergraph{A}{\veb}{\mathcal{R}}$ such that
the right-hand side components exceeds every given bound $N\in\Zge$.
\eoproof\end{proof}
 
\par{It is not hard to see that Theorem~\ref{thm:IndependenceofMarginSize}
is true if we choose Universal Gr\"obner bases or Graver
bases as sets of allowed moves as well. With a view towards
Conjecture~\ref{conj:Engstroem:2}: due to the isomorphism, all
properties of the underlying graph
carry over and hence it is enough to consider a Gr\"obner fiber graph
of a matrix whose connectivity is strictly less than its minimal
degree in a low-sized right-hand side (see
Section~\ref{sec:GroebnerFibers}). Since there are no general tools
for determining the connectivity of fiber graphs available, we
establish
some definitions and lemmas from which our connectivity results in
Section~\ref{sec:GraverFiber} benefit from. First, we slightly extend
our definition of a fiber graph in Section~\ref{sec:Introduction} in
the sense that we do not only restrict on fibers as a set of nodes but
rather on arbitrary sets of integer points. For two sets
$\mathcal{F}\subset\Zge^k$ and $\M\subset\Z^k$, 
$\normalgraph{\mathcal{F}}{\M}$ is the graph on $\mathcal{F}$ in
which two \vertices\ $\vev,\veu\in\F$ are adjacent if either
$\vev-\veu\in\M$ or $\veu-\vev\in\M$. Given an integer vector
$\vew\in\Zge^k$, the \emph{box} of $\vew$ is
\begin{equation*}
 \normalbox{\vew}:=[\vew_1]\times[\vew_2]\times \cdots \times[\vew_k]\subset\Zge^k
\end{equation*}
and the standard basis of $\Z^k$ is $\standardbasis{k}:=\{\vece_i: i\in[k]\}\subset\Zge^k$. The next lemma states that the \vertex-connectivity is best possible
in the graph $\normalgraph{\normalbox{\vew}}{\standardbasis{k}}$. We omit the proof since it can easily archived as a consequence of the node-version of Lemma~\ref{lem:EdgeConnectivity}. Recall that for $\vew\in\Z^k$ the support $\supp(\vew)\subset[k]$ is the set of indices of all non-zero coordinates of $\vew$. }

\begin{lemma}\label{lem:BlockGraph}
For $\vew\in\Zge^k$, the minimal degree and \vertex-connectivity of the graph $\normalgraph{\normalbox{\vew}}{\standardbasis{k}}$ 
equals $|\supp(\vew)|$. 
\end{lemma}

\par{In order to exploit a more refined structure of fiber graphs of
$A_k$ (see Section~\ref{sec:FiberStructure}), we first have a
look at sets of the following type: for a given set
$\F\subset\Z^k$ and a vector $\veb\in\Z^k$ the $\veb$-slack of
$\F$ is 
\begin{equation}\label{equ:SlackedBox}
\slacked{\F}{\veb}:=\left\{\begin{pmatrix}\vex\\\veb-\vex\end{pmatrix}:\vex\in\F\right\}\subset\Z^{2k}
\end{equation}
and the $\mathbf{0}_k$-slack of $\F$ is abbreviated as $\slackedmoves{\F}:=\slacked{\F}{\mathbf{0}_k}$. 
We need in Section~\ref{sec:Bases} and~\ref{sec:FiberStructure} the special case that $\F=\normalbox{\vew}$ and we denote its slack short by $\slackedbox{\vew}:=\slacked{\normalbox{\vew}}{\vew}$.
In the next lemma we show that the connectivity of a graph does not change by adding slacks to the set of \vertices\ if we slack the set of moves by $\mathbf{0}_k$, too.}

\begin{lemma}\label{lem:SlackedGraphs}
For $\veb\in\Zge^k$, $\F\subset\normalbox{\veb}$, and a set of moves $\M\subset\Z^k$ we have
\begin{equation}\label{equ:SlackedGraphIso}
\normalgraph{\F}{\M}\cong\normalgraph{\slacked{\F}{\veb}}{\slackedmoves{\M}}.
\end{equation}
\end{lemma}
\begin{proof}
Since $\F\subset\normalbox{\veb}$, we have $\slacked{\F}{\veb}\subset\Zge^{2k}$ and hence the graph on the right-hand side of \eqref{equ:SlackedGraphIso} is well-defined in the sense our definition given above. The map
\begin{equation*}
\mathcal{F}\rightarrow\slacked{\F}{\veb}, \vev\mapsto\begin{pmatrix}\vev\\\veb-\vev\end{pmatrix}
\end{equation*}
gives a bijection between the \vertices\ of the two graphs in \eqref{equ:SlackedGraphIso} which does not only respect the set of edges, but even more induces a bijection between them, too. \eoproof
\end{proof}

\section{Graver and Gr\"obner bases of $A_k$}\label{sec:Bases}

\par{In this section, we construct both the Graver basis and the reduced Gr\"obner basis with respect to a lexicographic term ordering of $A_k$ as defined in \eqref{def:Ak}.
For this, it is necessary to recall the definition of the Graver basis of a matrix first. 
Let $\sqsubseteq$ be the partial ordering on $\Z^n$ such that for two integer vectors $\veu,\vev\in\Z^n$ we have $\veu\sqsubseteq \vev$ if $\veu_i\cdot\vev_i\ge 0$ and $|\veu_i|\le|\vev_i|$ for all $i\in[n]$. 
The \emph{Graver basis} $\Graver(A)$ of a matrix $A\in\Z^{d\times n}$ is the set of all $\sqsubseteq$-minimal elements in $\ikern{A}{n}\setminus\{\mathbf{0}_n\}$.
Note that $\Graver(A)$ is always a finite set~\cite[Chapter~3]{Loera2013}.
When it comes to calculations of Graver bases, the following
definition is very helpful: for a non-negative vector $\vev\in\Zge^k$,
let $\chi(\vev)\in \{0,1\}^k$ be such that we have for all $i\in[k]$
\begin{equation*}
\chi(\vev)_i=\begin{cases}
0,&\textnormal{ if }\vev_i=0\\ 
1,&\textnormal{ if }\vev_i\neq 0\\ 
\end{cases}.
\end{equation*}}

\begin{theorem}\label{theorem:Graver basis of A_k}
For $k>0$, the Graver basis of $A_k$ is the (disjoint) union of
\begin{equation}\label{equ:GraverBasis}
\pm\slackedbox{-\mathbf{1}_k}\times\slackedbox{\mathbf{1}_k}\times\{-1\}\times\{1\} 
\end{equation}
and the sets
\begin{equation}\label{equ:GraverBasisStandardMoves}
\begin{split}
	&\pm\slackedmoves{\standardbasis{k}}\times\{\mathbf{0}_{2k}\}\times\{0\}\times\{0\} \textnormal{ and}\\
	&\pm\{\mathbf{0}_{2k}\}\times\slackedmoves{\standardbasis{k}}\times\{0\}\times\{0\}.
\end{split}
\end{equation}
\end{theorem}
\begin{proof}
Denote the union of the sets given in \eqref{equ:GraverBasis} and \eqref{equ:GraverBasisStandardMoves} by
$G$. We show that for every $\veu\in\Z^{4k+2}$ with $\veu\neq\mathbf{0}_{4k+2}$ and $A_k\veu=\mathbf{0}_{2k+1}$
there exists $\veg\in G$ such that $\veg\sqsubseteq \veu$. We write $\veu=\left(\vex_1, \vex_2, \vey_1,\vey_2,s,t\right)^\T$
for vectors $\vex_1,\vex_2,\vey_1,\vey_2\in\Z^k$ and integers $s,t\in\Z$. The
block structure of $A_k$ yields the following equations:
\begin{equation}\label{eq:GraverProof}
\begin{split}
\vex_1+\vex_2&=s\cdot\mathbf{1}_k \\
\vey_1+\vey_2&=t\cdot\mathbf{1}_k \\
s+t&=0.
\end{split}
\end{equation}
We distinguish the following two cases.

Case 1: $s=-t=0$. Clearly, we have $\vex_1=-\vex_2$ and $\vey_1=-\vey_2$. As
$\veu\neq \mathbf{0}_{4k+2}$ we can assume without loss of generality that $\vex_1\neq
\mathbf{0}_{k}$. Thus, there is $i\in [k]$ and $\lambda\in\{-1,1\}$ such that

\begin{equation*}
\lambda\cdot\begin{pmatrix}\vece_i\\-\vece_i\end{pmatrix}\sqsubseteq\begin{pmatrix}\vex_1
\\ \vex_2\end{pmatrix}
\end{equation*}
which gives rise to an element in $\slackedmoves{\standardbasis{k}}\times\{\mathbf{0}_{2k}\}\times\{0\}\times\{0\}$ which is less than $\veu$ with respect to $\sqsubseteq$.

Case 2: $s=-t\neq 0$. Without restricting generality (since $G$ is symmetric
we can multiply $\veu$ by $-1$ if necessary) we can assume that $t>0$
and as $t$ is an integer we have $t\ge 1$ and thus $s=-t\le -1$. 
Clearly, we have $-\vex_1^-\sqsubseteq\vex_1$ and $\vex_1^+\sqsubseteq\vex_1$ and hence we have $-\chi(\vex_1^-)\sqsubseteq\vex_1$. As $s\le -1$, equation \eqref{eq:GraverProof} gives  $-\mathbf{1}_k+\chi(\vex_1^-)\sqsubseteq\vex_2$ which implies
\begin{equation*}
\begin{pmatrix}-\chi(\vex_1^-) \\ -\mathbf{1}_k+\chi(\vex_1^-)\end{pmatrix}
\sqsubseteq\begin{pmatrix}\vex_1\\\vex_2\end{pmatrix}.
\end{equation*}
Similarly, one can show that
\begin{equation*}
\begin{pmatrix}\chi(\vey_1^+) \\ \mathbf{1}_k-\chi(\vey_1^+)\end{pmatrix}
\sqsubseteq\begin{pmatrix}\vey_1\\\vey_2\end{pmatrix}.
\end{equation*}
Since $-\chi(\vex_i^-)\in\normalbox{-\mathbf{1}_k}$ and $\chi(\vey_1^+)\in\normalbox{\mathbf{1}_k}$ and due to $s\le -1$ and $t\ge 1$ we found an element in $\slackedbox{-\mathbf{1}_k}\times\slackedbox{\mathbf{1}_k}\times\{-1\}\times\{1\}\subset G$ which is less than $\veu$ with respect to the partial ordering $\sqsubseteq$.
\eoproof\end{proof}

\par{In the following, we consider a Gr\"obner basis with respect to
	the \emph{lexicographic ordering} $\lex$ on $\Zge^{n}$ where
	for two integer vectors $\veu,\vev\in\Zge^{n}$ with
$\veu\neq\vev$ we have $\veu\lex\vev$ if $\veu_i<\vev_i$
for the smallest $i\in[n]$ such that $\veu_i\neq\vev_i$. The next
theorem extracts the reduced Gr\"obner basis of $A_k$ with respect to $\lex$ from its Graver basis.}

\begin{theorem}\label{theorem:ReducedGBToricIdeal}
For $k>0$, the reduced Gr\"obner basis of $A_k$ with respect to $\lex$ consists of the vector
\begin{equation*}
(\mathbf{0}_k,\mathbf{1}_k,\mathbf{0}_k,-\mathbf{1}_k,1,-1)^\T
\end{equation*}
and the vectors of the sets
\begin{equation}\label{equ:GroebnerBasis}
\begin{split}
	\slackedmoves{\standardbasis{k}}\times\{\mathbf{0}_{2k}\}&\times\{0\}\times\{0\}\textnormal{ and}\\
	\{\mathbf{0}_{2k}\}\times\slackedmoves{\standardbasis{k}}&\times\{0\}\times\{0\}.
\end{split}
\end{equation}\end{theorem}
\begin{proof}
As any reduced Gr\"obner basis of $A_k$ is contained in the Graver basis of $A_k$ \cite[Proposition~4.11]{Sturmfels1996}, the result follows immediately by extracting those elements from the Graver basis $\Graver(A_k)$, given in Theorem~\ref{theorem:Graver basis of A_k}, that cannot be reduced by other elements of $\Graver(A_k)$ with respect to $\lex$. \eoproof
\end{proof}

\section{The Fiber-Structure of $A_k$}\label{sec:FiberStructure}

\par{Equipped with explicit descriptions of both the Graver basis and the reduced $\lex$-Gr\"obner basis of $A_k$, we discover in this section the structure of $\fiber{A_k}{\veb}$ for any given right-hand side vector $\veb\in\Z^{2k+1}$. We write $\veb=(\vew_1,\vew_2,c)^\T\in\Z^{2k+1}$ with vectors $\vew_1,\vew_2\in\Z^k$ and $c\in\Z$. We assume that $\fiber{A_k}{\veb}\neq\emptyset$ and hence we can choose an arbitrary element
$\veu\in\fiber{A_k}{\veb}$ and write $\veu=(\vex_1,\vex_2,\vey_1,\vey_2,s,t)^\T\in\Zge^{4k+2}$
with vectors $\vex_1,\vex_2,\vey_1,\vey_2\in\Zge^k$ and $s,t\in\Zge$. Since we
have $A_k\veu=\veb$, we obtain the following relations:
\begin{equation}\label{equ:EquationsFiberElement}
\begin{split}
\vex_1+\vex_2&=\vew_1+s\cdot\mathbf{1}_k\\
\vey_1+\vey_2&=\vew_2+t\cdot\mathbf{1}_k\\
s+t&=c.
\end{split}
\end{equation}
We see immediately that we must have $\vew_1+s\cdot\mathbf{1}_k\ge\mathbf{0}_k$,
$\vew_2+t\cdot\mathbf{1}_k\ge\mathbf{0}_k$ and  $c\ge 0$, since otherwise $\fiber{A_k}{\veb}=\emptyset$. As $t$ is uniquely determined by $t=c-s$, those inequalities give 
\begin{equation}\label{equ:InequalitiesforS}
\underbrace{\max\{(\vew_1^-)_i: i\in[k]\}}_{=\|\vew_1^-\|_\infty}\le s\le
c-\underbrace{\max\{(\vew_2^-)_i: i\in[k]\}}_{=\|\vew_2^-\|_\infty}.
\end{equation}
So we can define both a lower and an upper bound on $s$ by
\begin{equation*}
l(\veb):=\|\vew_1^-\|_\infty\textnormal{ and }u(\veb):=c-\|\vew_2^-\|_\infty.
\end{equation*}
If $l(\veb)>u(\veb)$, we certainly have $\fiber{A_k}{\veb}=\emptyset$ and hence we can assume that $l(\veb)\le u(\veb)$. The equations in \eqref{equ:EquationsFiberElement} suggest that we can regard $\vex_2$ and $\vey_2$ as slack variables since they are already uniquely determined by the choices of $\vex_1\in\normalbox{\vew_1+s\cdot\mathbf{1}_k}$ and $\vey_1\in\normalbox{\vew_2+(c-s)\cdot\mathbf{1}_k}$. Hence, any element of the fiber looks like
\begin{equation}\label{eq:fibervert}
\fibervert{\vex}{\vey}{s}:=\begin{pmatrix}
\vex \\
\vew_1+s\cdot\mathbf{1}_k -\vex\\
\vey \\
\vew_2+(c-s)\cdot\mathbf{1}_k -\vey\\
s\\
c-s\\
\end{pmatrix}
\end{equation}
for $\vex\in\normalbox{\vew_1+s\cdot\mathbf{1}_k}$ and $\vey\in\normalbox{\vew_2+(c-s)\cdot\mathbf{1}_k}$. Using our definition of slacked boxes as defined in \eqref{equ:SlackedBox}, we obtain an explicit description of elements in $\fiber{A_k}{\veb}$ which have their $(4k+1)$th coordinate equal to $s$:
\begin{equation}\label{equ:BlockRepresentation}
\cbox{s}{\veb}:=\slackedbox{\vew_1+s\cdot\mathbf{1}_k}\times\slackedbox{\vew_2+(c-s)\cdot\mathbf{1}_k}\times\{s\}\times\{c-s\}\subset\Zge^{2k+2k+2}.
\end{equation}
This gives us a very convenient partition of the fiber into $u(\veb)-l(\veb)+1$ disjoint sets:
\begin{equation}\label{equ:FiberPresentation}
\fiber{A_k}{\veb}=\bigcup_{s=l(\veb)}^{u(\veb)}\cbox{s}{\veb}.
\end{equation}
We see that the Graver moves from the sets defined in \eqref{equ:GraverBasis} connect \vertices\ from two adjacent boxes $\cbox{s_1}{\veb}$ and $\cbox{s_2}{\veb}$ with $|s_1-s_2|=1$, whereas Graver moves from \eqref{equ:GraverBasisStandardMoves} connect \vertices\ within the same box $\cbox{s}{\veb}$ (see Figure~\ref{fig:MoveTypes}). 

\begin{figure}[tbh]
	\centering
				\begin{tikzpicture}[scale=0.3]

\foreach \cube in {0,1} {

 \foreach \numberX in {0,1,2}{
	\foreach \numberY in {0,1,2}{
	
 	\node[draw,circle,fill=black,inner sep=0pt,minimum size=4pt,xshift=2*\cube cm] (\cube-0-\numberX-\numberY) at 	(2*\numberX,2*\numberY) {};	
	
	}}

 \foreach \numberX in {0,1,2}{
	\foreach \numberY in {0,1,2}{

 	\node[draw,circle,fill=black,inner sep=0pt,minimum size=4pt,xshift=0.2cm+2*\cube cm,yshift=0.2cm] (\cube-1-\numberX-\numberY) at 	(2*\numberX,2*\numberY) {};

	}}

 \foreach \numberX in {0,1,2}{
	\foreach \numberY in {0,1,2}{

 	\node[draw,circle,fill=black,inner sep=0pt,minimum size=4pt,xshift=0.4cm+2*\cube cm,yshift=0.4cm] (\cube-2-\numberX-\numberY) at 	(2*\numberX,2*\numberY) {};

	}}

\foreach \number in {0,1,2} {

   	\path (\cube-\number-0-0) edge[font=\small,dotted] (\cube-\number-0-1);
   	\path (\cube-\number-0-1) edge[font=\small,dotted] (\cube-\number-0-2);
	\path (\cube-\number-0-2) edge[font=\small,dotted] (\cube-\number-1-2);
	\path (\cube-\number-1-2) edge[font=\small,dotted] (\cube-\number-2-2);
	\path (\cube-\number-2-2) edge[font=\small,dotted] (\cube-\number-2-1);
	\path (\cube-\number-2-1) edge[font=\small,dotted] (\cube-\number-2-0);
	\path (\cube-\number-2-0) edge[font=\small,dotted] (\cube-\number-1-0);
	\path (\cube-\number-1-0) edge[font=\small,dotted] (\cube-\number-0-0);
	
	\path (\cube-\number-1-1) edge[font=\small,dotted] (\cube-\number-1-2);
	\path (\cube-\number-1-1) edge[font=\small,dotted] (\cube-\number-2-1);
	\path (\cube-\number-1-1) edge[font=\small,dotted] (\cube-\number-0-1);
	\path (\cube-\number-1-1) edge[font=\small,dotted] (\cube-\number-1-0);

}

 \foreach \numberX in {0,1,2}{
	\foreach \numberY in {0,1,2}{
	\path (\cube-0-\numberX-\numberY) edge[font=\small,dotted] (\cube-2-\numberX-\numberY);

}}

}
	
	\path (0-1-2-1) edge[font=\small,color=red] (1-1-0-0);
	\path (1-0-1-1) edge[font=\small,color=red] (1-0-2-1);
		
\end{tikzpicture}
	\caption{Different types of Graver moves of $A_k$.}\label{fig:MoveTypes}
\end{figure}
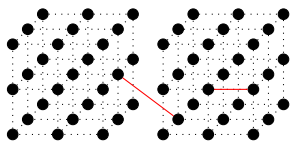

Even more, since the $(4k+1)$th and $(4k+2)$th coordinates coincide for all elements in $\cbox{s}{\veb}$, the next lemma follows immediately.}
\begin{lemma}\label{lem:BoxSubgraphsCoincide}
For $\veb\in\Z^{2k+1}$ and $s\in[l(\veb),u(\veb)]$, the following equality holds: 
\begin{equation*}
\normalgraph{\cbox{s}{\veb}}{\Graver(A_k)}=\normalgraph{\cbox{s}{\veb}}{\mathcal{R}_{\lex}(A_k)}.
\end{equation*}
\end{lemma}

\par{Based on our observations in Section~\ref{sec:Connectivity}, we know that the \vertex-connectivity in those induced subgraphs is best possible as the next lemma shows.}

\begin{lemma}\label{lem:PathsInbox}
Let $\veb\in\Z^{2k+1}$ such that $\fiber{A_k}{\veb}\neq\emptyset$. For all $s\in[l(\veb),u(\veb)]$, the minimal degree and the \vertex-connectivity of the graph $\normalgraph{\cbox{\veb}{s}}{\Graver(A_k)}$ equal
\begin{equation*}
|\supp(\vew_1+s\cdot\mathbf{1}_k)|+|\supp(\vew_2+(c-s)\cdot\mathbf{1}_k)|.
\end{equation*}
\end{lemma}

\begin{proof}
Using the representation of $\cbox{s}{\veb}$ in \eqref{equ:BlockRepresentation} and a projection onto the first $4k$ coordinates, we obtain that the induced subgraph of $\fibergraph{A_k}{\veb}{\Graver(A_k)}$ on the \vertices\ $\cbox{s}{\veb}$ is isomorphic to the graph
\begin{equation}\label{equ:InducedSubgraph}
	\normalgraph{\slackedbox{\vew_1+s\cdot\mathbf{1}_k}\times\slackedbox{\vew_2+(c-s)\cdot\mathbf{1}_k}}{\slackedmoves{\standardbasis{k}}\times\{\mathbf{0}_{2k}\}\cup\{\mathbf{0}_{2k}\}\times\slackedmoves{\standardbasis{k}}}.
\end{equation}
Graphs of this particular structure can be interpreted as the \emph{Cartesian product} of two related graphs, in our case here, $\normalgraph{\slackedbox{\vew_1+s\cdot\mathbf{1}_k}}{\slackedmoves{\standardbasis{k}}}$ and $\normalgraph{\slackedbox{\vew_2+(c-s)\cdot\mathbf{1}_k}}{\slackedmoves{\standardbasis{k}}}$ (we refer to \cite{W.S.Chiue1999} for a definition).
This gives that the minimal degree of this graph is the sum of the minimal degrees of $\normalgraph{\slackedbox{\vew_1+s\cdot\mathbf{1}_k}}{\slackedmoves{\standardbasis{k}}}$ and of $\normalgraph{\slackedbox{\vew_2+(c-s)\cdot\mathbf{1}_k}}{\slackedmoves{\standardbasis{k}}}$. Using the isomorphism given in Lemma~\ref{lem:SlackedGraphs}, their minimal degrees coincide with the minimal degrees of the graphs $\normalgraph{\normalbox{\vew_1+s\cdot\mathbf{1}_k}}{\standardbasis{k}}$ and $\normalgraph{\normalbox{\vew_2+(c-s)\cdot\mathbf{1}_k}}{\standardbasis{k}}$, respectively. Applying the formula of Lemma~\ref{lem:BlockGraph}, the minimal degree of the graph given in \eqref{equ:InducedSubgraph} equals
\begin{equation*}
|\supp(\vew_1+s\cdot\mathbf{1}_k)|+|\supp(\vew_2+(c-s)\cdot\mathbf{1}_k)|.
\end{equation*}
as claimed.\eoproof
\end{proof}

\par{Whereas Lemma~\ref{lem:BoxSubgraphsCoincide} states that the Gr\"obner and Graver fiber graphs coincide on the subgraph induced by $\cbox{s}{\veb}$, Lemma~\ref{lem:PathsInbox} says that the \vertex-connectivity in those subgraphs is best possible. But what about moves between two neighbouring boxes of $\fiber{A_k}{\veb}$? Let us now determine under which conditions \vertices\ of neighboring boxes are adjacent to each other. For that it is necessary that $\fiber{A_k}{\veb}$ has at least two boxes, which is precisely the case if $l(\veb)<u(\veb)$. To simplify our proofs it is reasonable to define for all choices $\vev_1,\vev_2\in\normalbox{\mathbf{1}_k}$ the following move from the Graver basis of $A_k$:
\begin{equation*}
\gr{\vev_1}{\vev_2}:=
\begin{pmatrix}
-\vev_1 \\
-\mathbf{1}_k+\vev_1 \\
\vev_2 \\
\mathbf{1}_k-\vev_2 \\
-1\\
1
\end{pmatrix}\in\slackedbox{-\mathbf{1}_k}\times\slackedbox{\mathbf{1}_k}\times\{-1\}\times\{1\}.
\end{equation*}
Choose $s\in [l(\veb),u(\veb)]$ and let 
$(\vex,\vey)^\T\in\normalbox{\vew_1+\mathbf{1}_k\cdot s}\times\normalbox{\vew_2+(c-s)\cdot\mathbf{1}_k}$ and $\vev_1,\vev_2\in\normalbox{\mathbf{1}_k}$. The following conditions on $\vev_1$ and $\vev_2$
\begin{subequations}
\begin{align}
\supp(\vev_1)\subset\supp(\vex)&\textnormal{ and } [k]\setminus\supp(\vev_1)\subset\supp(\vew_1+s\cdot\mathbf{1}_k-\vex)\label{equ:EdgePlusGraver}\\
\supp(\vev_2)\subset\supp(\vey)&\textnormal{ and } [k]\setminus\supp(\vev_2)\subset\supp(\vew_2+(c-s)\cdot\mathbf{1}_k-\vey)\label{equ:EdgeMinusGraver}
\end{align}
\end{subequations}
lead to a technical characterization for a Graver move to be applicable at $\fibervert{\vex}{\vey}{s}$:
\begin{equation*}
\begin{split}
\fibervert{\vex}{\vey}{s}\xlongleftrightarrow{\gr{\vev_1}{\vev_2}}\fibervert{\vex-\vev_1}{\vey+\vev_2}{s-1}
&\iff \eqref{equ:EdgePlusGraver}\textnormal{ and } s>l(\veb)\\
\fibervert{\vex}{\vey}{s}\xlongleftrightarrow{-\gr{\vev_1}{\vev_2}}\fibervert{\vex+\vev_1}{\vey-\vev_2}{s+1}
&\iff \eqref{equ:EdgeMinusGraver}\textnormal{ and } s<u(\veb).
\end{split}
\end{equation*}

In particular, we see that only a fraction of moves between two adjacent boxes of $\fiber{A_k}{\veb}$ are actually moves from the lexicographic Gr\"obner basis of $A_k$. So the main difference of the fiber graphs of $A_k$ with respect to Graver and Gr\"obner moves results from how the boxes $\cbox{s}{\veb}$ are connected among each other. From our observations in this section, we obtain that there is a large number of Graver moves between two neighboring boxes and we summarize this results in the following proposition.}

\begin{proposition}\label{prop:GraverDegree}
Let $\veb=(\vew_1,\vew_2,c)^\T\in\Z^{2k+1}$ with $\vew_1,\vew_2\in\Z^k$ and
$c\in\Z$ such that $\fiber{A_k}{\veb}\neq\emptyset$ and consider the fiber graph $\fibergraph{A_k}{\veb}{\Graver(A_k)}$. 
For $s\in[l(\veb),u(\veb)]$, a \vertex\ $\vev\in\cbox{s}{\veb}$ has
neighbors in $\cbox{s-1}{\veb}$ if and only if $s>l(\veb)$ and in this case that are at least
$2^k$ many. In the same way, $\vev$ has neighbors in
$\cbox{s+1}{\veb}$ if and only if $s<u(\veb)$ and that are at least
$2^k$ many in this case.
\end{proposition}
\begin{proof}
The statement of the proposition follows immediately from the fact that moves of the form $\gr{\chi(\vex)}{\vev_2}$ are applicable at $\fibervert{\vex}{\vey}{s}$ for all $\vev_2\in\normalbox{\mathbf{1}_k}$ if $s>l(\veb)$ and in the same way we see that moves of the form $-\gr{\vev_1}{\chi(\vey)}$ are applicable at $\fibervert{\vex}{\vey}{s}$ if $s<u(\veb)$ for all $\vev_1\in\normalbox{\mathbf{1}_k}$.\eoproof
\end{proof}

\section{Gr\"obner Fiber Graphs of $A_k$}\label{sec:GroebnerFibers}

\par{As mentioned in the previous section, the number of edges between two boxes
of a fiber is significantly higher under the Graver basis than the
reduced lexicographic Gr\"obner basis and
our hope is that this affects the connectivity of the fiber graphs.
Indeed, considering the fiber of $\vece_{2k+1}\in\Z^{2k+1}$, we have that
$l(\vece_{2k+1})=0$ and $u(\vece_{2k+1})=1$. Thus,
\eqref{equ:FiberPresentation} gives
\begin{equation*}
\begin{split}
\fiber{A_k}{\vece_{2k+1}}&=\cbox{0}{\vece_{2k+1}}\cup\cbox{1}{\vece_{2k+1}}\\
&=\slackedbox{\mathbf{0}_k}\times\slackedbox{\mathbf{1}_k}\times\{0\}\times\{1\}\cup\slackedbox{\mathbf{1}_k}\times\slackedbox{\mathbf{0}_k}\times\{1\}\times\{0\}\\
&=\{\mathbf{0}_k\}\times\slackedbox{\mathbf{1}_k}\times\{0\}\times\{1\}\cup\slackedbox{\mathbf{1}_k}\times\{\mathbf{0}_k\}\times\{1\}\times\{0\}.
\end{split}
\end{equation*}
This combined with Lemma~\ref{lem:PathsInbox} implies that the minimal degree of $\fibergraph{A_k}{\vece_{2k+1}}{\GB}$ is at least $k$. Due to the connection to slacked boxes, Lemma~\ref{lem:PathsInbox} explains the structure of the fiber within a box  very well. But what about edges between two boxes with respect to Gr\"obner moves? According to Theorem~\ref{theorem:ReducedGBToricIdeal}, the only
move available is
\begin{equation*} 
\gr{\mathbf{0}_k}{\mathbf{0}_k}=
 (\mathbf{0}_k,-\mathbf{1}_k,\mathbf{0}_k,\mathbf{1}_k,\mathbf{0}_k,-1,1)^\T
\end{equation*}
 and according to Section~\ref{sec:FiberStructure}, this move can be
 applied only once in the fiber $\fiber{A_K}{\vece_{2k+1}}$, namely as move between the following \vertices:
\begin{equation*}
(\mathbf{0}_k,\mathbf{0}_k,\mathbf{0}_k,\mathbf{1}_k,0,1)^\T
\xlongleftrightarrow{\gr{\mathbf{0}_k}{\mathbf{0}_k}}
(\mathbf{0}_k,\mathbf{1}_k,\mathbf{0}_k,\mathbf{0}_k,1,0)^\T.
\end{equation*}
This means there is only a single edge connecting $\cbox{0}{\veb}$ and
$\cbox{1}{\veb}$ (see Figure~\ref{fig:Counterexample}) and hence the minimal degree of $\fibergraph{A_k}{\vece_{2k+1}}{\GB}$ equals $k$. 
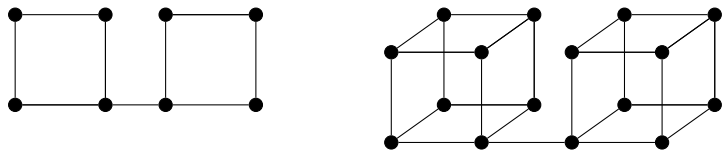
\begin{figure}[tbh]
	\centering
		\begin{tikzpicture}[yscale=0.4,xscale=0.4]

 	\node[draw,circle,fill=black,xshift=-5cm,yshift=0.5cm,inner sep=0pt,minimum size=5pt] (A-0-1-0-0) at 	(0,0) {};
 	\node[draw,circle,fill=black,xshift=-5cm,yshift=0.5cm,inner sep=0pt,minimum size=5pt] (A-0-1-0-1) at 	(0,3) {};
 	\node[draw,circle,fill=black,xshift=-5cm,yshift=0.5cm,inner sep=0pt,minimum size=5pt] (A-0-1-1-0) at 	(3,0) {};
 	\node[draw,circle,fill=black,xshift=-5cm,yshift=0.5cm,inner sep=0pt,minimum size=5pt] (A-0-1-1-1) at 	(3,3) {};

   \path (A-0-1-0-0) edge[font=\small] (A-0-1-0-1);
   \path (A-0-1-0-0) edge[font=\small] (A-0-1-1-0);

   \path (A-0-1-1-0) edge[font=\small] (A-0-1-0-0);
   \path (A-0-1-1-0) edge[font=\small] (A-0-1-1-1);

   \path (A-0-1-1-1) edge[font=\small] (A-0-1-0-1);
   \path (A-0-1-1-1) edge[font=\small] (A-0-1-1-0);

 	\node[draw,circle,fill=black,xshift=-3cm,yshift=0.5cm,inner sep=0pt,minimum size=5pt] (A-1-0-0-0) at 	(0,0) {};
 	\node[draw,circle,fill=black,xshift=-3cm,yshift=0.5cm,inner sep=0pt,minimum size=5pt] (A-1-0-0-1) at 	(0,3) {};
 	\node[draw,circle,fill=black,xshift=-3cm,yshift=0.5cm,inner sep=0pt,minimum size=5pt] (A-1-0-1-0) at 	(3,0) {};
 	\node[draw,circle,fill=black,xshift=-3cm,yshift=0.5cm,inner sep=0pt,minimum size=5pt] (A-1-0-1-1) at 	(3,3) {};

   \path (A-1-0-0-0) edge[font=\small] (A-1-0-0-1);
   \path (A-1-0-0-0) edge[font=\small] (A-1-0-1-0);

   \path (A-1-0-0-1) edge[font=\small] (A-1-0-1-1);

   \path (A-1-0-1-1) edge[font=\small] (A-1-0-1-0);
   \path (A-1-0-1-1) edge[font=\small] (A-1-0-0-1);

   \path (A-0-1-1-0) edge[font=\small] (A-1-0-0-0);

 	\node[draw,circle,fill=black,inner sep=0pt,minimum size=5pt] (0-0-0-0) at 	(0,0) {};
 	\node[draw,circle,fill=black,inner sep=0pt,minimum size=5pt] (0-0-0-1) at 	(0,3) {};
 	\node[draw,circle,fill=black,inner sep=0pt,minimum size=5pt] (0-0-1-0) at 	(3,0) {};
 	\node[draw,circle,fill=black,inner sep=0pt,minimum size=5pt] (0-0-1-1) at 	(3,3) {};

 	\node[draw,circle,fill=black,xshift=0.7cm,yshift=0.5cm,inner sep=0pt,minimum size=5pt] (0-1-0-0) at 	(0,0) {};
 	\node[draw,circle,fill=black,xshift=0.7cm,yshift=0.5cm,inner sep=0pt,minimum size=5pt] (0-1-0-1) at 	(0,3) {};
 	\node[draw,circle,fill=black,xshift=0.7cm,yshift=0.5cm,inner sep=0pt,minimum size=5pt] (0-1-1-0) at 	(3,0) {};
 	\node[draw,circle,fill=black,xshift=0.7cm,yshift=0.5cm,inner sep=0pt,minimum size=5pt] (0-1-1-1) at 	(3,3) {};

   \path (0-0-0-0) edge[font=\small] (0-0-0-1);
   \path (0-0-0-0) edge[font=\small] (0-0-1-0);
   \path (0-0-0-0) edge[font=\small] (0-1-0-0);

   \path (0-0-0-1) edge[font=\small] (0-1-0-1);
   \path (0-0-0-1) edge[font=\small] (0-0-1-1);

   \path (0-1-0-0) edge[font=\small] (0-1-0-1);
   \path (0-1-0-0) edge[font=\small] (0-1-1-0);

   \path (0-1-1-0) edge[font=\small] (0-0-1-0);
   \path (0-1-1-0) edge[font=\small] (0-1-0-0);
   \path (0-1-1-0) edge[font=\small] (0-1-1-1);

   \path (0-0-1-1) edge[font=\small] (0-1-1-1);
   \path (0-0-1-1) edge[font=\small] (0-0-1-0);
   \path (0-0-1-1) edge[font=\small] (0-0-0-1);

   \path (0-1-1-1) edge[font=\small] (0-0-1-1);
   \path (0-1-1-1) edge[font=\small] (0-1-0-1);
   \path (0-1-1-1) edge[font=\small] (0-1-1-0);

 	\node[draw,circle,fill=black,xshift=2.4cm,yshift=0cm,inner sep=0pt,minimum size=5pt] (1-0-0-0) at 	(0,0) {};
 	\node[draw,circle,fill=black,xshift=2.4cm,yshift=0cm,inner sep=0pt,minimum size=5pt] (1-0-0-1) at 	(0,3) {};
 	\node[draw,circle,fill=black,xshift=2.4cm,yshift=0cm,inner sep=0pt,minimum size=5pt] (1-0-1-0) at 	(3,0) {};
 	\node[draw,circle,fill=black,xshift=2.4cm,yshift=0cm,inner sep=0pt,minimum size=5pt] (1-0-1-1) at 	(3,3) {};

 	\node[draw,circle,fill=black,xshift=3.1cm,yshift=0.5cm,inner sep=0pt,minimum size=5pt] (1-1-0-0) at 	(0,0) {};
 	\node[draw,circle,fill=black,xshift=3.1cm,yshift=0.5cm,inner sep=0pt,minimum size=5pt] (1-1-0-1) at 	(0,3) {};
 	\node[draw,circle,fill=black,xshift=3.1cm,yshift=0.5cm,inner sep=0pt,minimum size=5pt] (1-1-1-0) at 	(3,0) {};
 	\node[draw,circle,fill=black,xshift=3.1cm,yshift=0.5cm,inner sep=0pt,minimum size=5pt] (1-1-1-1) at 	(3,3) {};

   \path (1-0-0-0) edge[font=\small] (1-0-0-1);
   \path (1-0-0-0) edge[font=\small] (1-0-1-0);
   \path (1-0-0-0) edge[font=\small] (1-1-0-0);

   \path (1-0-0-1) edge[font=\small] (1-1-0-1);
   \path (1-0-0-1) edge[font=\small] (1-0-1-1);

   \path (1-1-0-0) edge[font=\small] (1-1-0-1);
   \path (1-1-0-0) edge[font=\small] (1-1-1-0);

   \path (1-1-1-0) edge[font=\small] (1-0-1-0);
   \path (1-1-1-0) edge[font=\small] (1-1-0-0);
   \path (1-1-1-0) edge[font=\small] (1-1-1-1);

   \path (1-0-1-1) edge[font=\small] (1-1-1-1);
   \path (1-0-1-1) edge[font=\small] (1-0-1-0);
   \path (1-0-1-1) edge[font=\small] (1-0-0-1);

   \path (1-1-1-1) edge[font=\small] (1-0-1-1);
   \path (1-1-1-1) edge[font=\small] (1-1-0-1);
   \path (1-1-1-1) edge[font=\small] (1-1-1-0);

   \path (0-0-1-0) edge[font=\small] (1-0-0-0);

\end{tikzpicture}
		\caption{A sketch of $\cbox{0}{\vece_{2k+1}}$ and
		$\cbox{1}{\vece_{2k+1}}$  for $k=2$ and $k=3$ with respect to Gr\"obner moves.}\label{fig:Counterexample}
\end{figure}
Thus, removing this edge gives a non-connected graph, i.e., the edge-connectivity of the fiber graph equals $1$. Since in all graphs the \vertex-connectivity is always less than the
edge-connectivity, we obtain the following corollary.}

\begin{corollary}\label{cor:CEconj1}
For $k>0$, the edge-connectivity of the fiber graph
$$\fibergraph{A_k}{\vece_{2k+1}}{\GB}$$ equals $1$, whereas its
minimal degree equals $k$. In particular, $A_k$ gives a
counterexample to Conjecture~\ref{conj:Engstroem:1} for $k\ge 2$.
\end{corollary}

\par{However, a priori $A_k$ does not provide evidence against Conjecture~\ref{conj:Engstroem:2} since
the conjecture claims that the \vertex-connectivity equals the minimal degree
only for sufficiently large right-hand sides. But Theorem~\ref{thm:IndependenceofMarginSize}
gives us an instruction how to modify $A_k$ such that it becomes a
counterexample to Conjecture~\ref{conj:Engstroem:2} as well:

\begin{equation*}
 B_k:=\begin{pmatrix} A_{k+1} & I_{2k+1} \\ \mathbf{0} &
I_{2k+1}\end{pmatrix}\in\Z^{(6k+3)\times(4k+2)}.
 \end{equation*} 
 }
\begin{corollary}
For $k>0$, there exists a term ordering $\prec_k$ on $\Zge^{6k+3}$ such
that for all $N\in\Zge$ there exists $\veb\in\Z^{4k+2}$ with $\veb\ge N\cdot\mathbf{1}_{4k+2}$
such that the edge-connectivity of
$\fibergraph{B_k}{\veb}{\mathcal{R}_{\prec_k}(B_k)}$ equals $1$
whereas its minimal degree equals $k$. In particular, $B_k$ gives a
counterexample to Conjecture~\ref{conj:Engstroem:2} for $k\ge 2$.
\end{corollary}

\section{Graver Fiber Graphs of $A_k$}\label{sec:GraverFiber}

\par{As shown in the last section, \vertex-connectivity and even edge-connectivity fail to be best possible in general in Gr\"obner fiber graphs. As the number of moves in the Graver basis enlarge the number of moves in a Gr\"obner basis by far, we hope that this circumstance reflects positively onto the connectivity of those fiber graphs. So let us now investigate how the situation looks like if we replace Gr\"obner moves with Graver moves. We prove that even if the edge-connectivity of some Gr\"obner fiber graphs of $A_k$ is rather bad, the
edge-connectivity of its Graver fiber graphs is best possible. With Proposition~\ref{prop:GraverDegree} in mind, let us first determine the minimal degree of the Graver fiber graphs.}
\begin{proposition}[Minimal degree]\label{prop:MinDeg}
Let $\veb=(\vew_1,\vew_2,c)^\T\in\Z^{2k+1}$ with $\vew_1,\vew_2\in\Z^k$ and
$c\in\Z$. If $l(\veb)=u(\veb)$, then we have
\begin{equation}\label{equ:MinDegbox}
\delta(\fibergraph{A_k}{\veb}{\Graver(A_k)})=|\supp(\vew_1+\|\vew_1^-\|_\infty\cdot\mathbf{1}_k)|+|\supp(\vew_2+\|\vew_2^-\|_\infty\cdot\mathbf{1}_k)|.
\end{equation}
Otherwise, if $l(\veb)<u(\veb)$, then we have
\begin{equation}\label{equ:MinDegNeiborboxes}
\delta(\fibergraph{A_k}{\veb}{\Graver(A_k)})=\min_{j\in\{1,2\}}\{|\supp(\vew_j+\|\vew_j^-\|_\infty\cdot\mathbf{1}_k)|\}+k+2^k.
\end{equation}
\end{proposition}
\begin{proof}
If $l(\veb)=s=u(\veb)$, the first statement is a reformulation of Lemma~\ref{lem:PathsInbox} due to $\fiber{A_k}{\veb}=\cbox{s}{\veb}$.
So assume that we have $l(\veb)<u(\veb)$. Since $s\in[l(\veb),u(\veb)]$, we must have either $s>l(\veb)$ or $s<u(\veb)$ and hence we have either $s>\|\vew_1^-\|_\infty$ of $c-s>\|\vew_2^-\|_\infty$. 
Putting those inequalities into the equation for the minimal degree in Lemma~\ref{lem:PathsInbox}, we obtain that a \vertex\ in $\cbox{s}{\veb}$ has at least
\begin{equation*}
\min_{j\in\{1,2\}}\{|\supp(\vew_j+\|\vew_j^-\|_\infty\cdot\mathbf{1}_k)|\}+k
\end{equation*}
neighbors in his own box $\cbox{s}{\veb}$. Furthermore, due to Proposition~\ref{prop:GraverDegree} and since either $s>l(\veb)$ or
$s<u(\veb)$, a \vertex\ in $\cbox{s}{\veb}$ has either at least $2^k$ neighbors in
$\cbox{s-1}{\veb}$ or at least $2^k$ neighbors in $\cbox{s+1}{\veb}$. This shows
that the minimal degree of $\fibergraph{A_k}{\veb}{\Graver(A_k)}$ is greater or equal than the
right-hand side of the term given in \eqref{equ:MinDegNeiborboxes}. Clearly, the
\vertex\ with minimal degree has to be either in $\cbox{l(\veb)}{\veb}$ or in
$\cbox{u(\veb)}{\veb}$. Thus, either
\begin{equation*}
\begin{pmatrix}
\vew_1+l(\veb)\cdot\mathbf{1}_k\\
\mathbf{0}_k \\
\vew_2+(c-l(\veb))\cdot\mathbf{1}_k\\
\mathbf{0}_k \\
l(\veb)\\
c-l(\veb)
\end{pmatrix}\in\cbox{l(\veb)}{\veb}\textnormal{ or }
\begin{pmatrix}
\vew_1+u(\veb)\cdot\mathbf{1}_k\\
\mathbf{0}_k \\
\vew_2+(c-u(\veb))\cdot\mathbf{1}_k\\
\mathbf{0}_k \\
u(\veb)\\
c-u(\veb)
\end{pmatrix}\in\cbox{u(\veb)}{\veb}
\end{equation*}
has the smallest degree in $\fibergraph{A_k}{\veb}{\Graver(A_k)}$.
\eoproof
\end{proof}

\par{With an explicit formula for the minimal degree of
$\fibergraph{A_k}{\veb}{\Graver(A_k)}$ in mind we can determine the
edge-connectivity of those fiber graphs explicitly. 
First, we consider edges between two neighboring boxes and we show that we find a suitable number of disjoint paths connecting their end-points. 
Please note that we make these paths even \vertex-disjoint in this case.}
\begin{lemma}[Edges within Box]\label{lemma:ReplacingEdgesinbox}
Let $\veb\in\Z^{2k+1}$ and $s\in[l(\veb),u(\veb)]$. Then for any two
adjacent \vertices\ in $\cbox{s}{\veb}$ there exist
$\delta(\fibergraph{A_k}{\veb}{\Graver(A_k)})$ many \vertex-disjoint paths in
$\fibergraph{A_k}{\veb}{\Graver(A_k)}$ connecting them.
\end{lemma}
\begin{proof}
We write $\veb=(\vew_1,\vew_2,c)^\T\in\Z^{2k+1}$ with $\vew_1,\vew_2\in\Z^k$ and
$c\in\Z$. Since we have $\fiber{A_k}{\veb}\neq\emptyset$ by assumption, we must
have $l(\veb)\le u(\veb)$. 
Due to Lemma~\ref{lem:PathsInbox} and Proposition~\ref{prop:MinDeg} there is nothing to show for $l(\veb)=u(\veb)$ and hence we assume that $l(\veb)<u(\veb)$. 
Without restricting generality, the two adjacent \vertices\ we need to connect with a sufficient number of \vertex-disjoint paths look like
\begin{equation}\label{equ:StartEdgeWithinBox}
\fibervert{\vex}{\vey}{s}\longleftrightarrow\fibervert{\vex+\vece_j}{\vey}{s}
\end{equation}
with $j\in[n]$, $s\in[l(\veb),u(\veb)]$, $\vex\in\normalbox{\vew_1+s\cdot\mathbf{1}_k}$, and $\vey\in\normalbox{\vew_2+(c-s)\cdot\mathbf{1}_k}$.
By Lemma~\ref{lem:PathsInbox} we find $|\supp(\vew_1+s\cdot\mathbf{1}_k)|+|\supp(\vew_2+(c-s)\cdot\mathbf{1}_k)|$
\vertex-disjoint paths connecting $\fibervert{\vex}{\vey}{s}$ and $\fibervert{\vex+\vece_j}{\vey}{s}$ which only use \vertices\ in $\cbox{s}{\veb}$. If we have $s>l(\veb)$, we define
\begin{equation*}
\vev_1:=
\begin{cases}
\chi(\vex)-\vece_j,&\textnormal{ if } \vex_j>0 \\
\chi(\vex),&\textnormal{ if } \vex_j=0 \\
\end{cases}\textnormal{ and }\vev_1':=
\begin{cases}
\chi(\vex),&\textnormal{ if } \vex_j>0 \\
\chi(\vex)+\vece_j,&\textnormal{ if } \vex_j=0 \\
\end{cases}.
\end{equation*}
Then we have $\vev_1\in\normalbox{\mathbf{1}_k}$ and $\vev_1'\in\normalbox{\mathbf{1}_k}$. Since we have by \eqref{equ:StartEdgeWithinBox} that $\vex+\vece_j\le\vew_1+s\cdot\mathbf{1}_k$, it is easy to see that $\vev_1$ fulfills \eqref{equ:EdgePlusGraver} and hence the Graver move $\gr{\vev_1}{\vev}$ is applicable at $\fibervert{\vex}{\vey}{s}$ for every $\vev\in\normalbox{\mathbf{1}_k}$. As $\vex-\vev_1+\vev_1'=\vex+\vece_j$ by construction, this gives for every $\vev\in\normalbox{\mathbf{1}_k}$ a path 
\begin{equation*}
\begin{split}
&\fibervert{\vex}{\vey}{s}\\
\longleftrightarrow&\fibervert{\vex-\vev_1}{\vey+\vev}{s-1}\in\cbox{s-1}{\veb}\\
\longleftrightarrow&\fibervert{\vex-\vev_1+\vev_1'}{\vey+\vev_2-\vev}{s-1+1}=\fibervert{\vex+\vece_j}{\vey}{s}
\end{split}
\end{equation*}
which only uses edges with end-points $\cbox{s}{\veb}$ and $\cbox{s-1}{\veb}$. On the other hand, if we
have $s<u(\veb)$, we have for every $\vev\in\normalbox{\mathbf{1}_k}$ a path
\begin{equation*}
\begin{split}
&\fibervert{\vex}{\vey}{s}\\
\longleftrightarrow&\fibervert{\vex+\vev}{\vey-\chi(\vey)}{s+1}\in\cbox{s+1}{\veb}\\
\longleftrightarrow&\fibervert{\vex+\vev+\vece_j}{\vey-\chi(\vey)}{s+1}\in\cbox{s+1}{\veb}\\
\longleftrightarrow&\fibervert{\vex+\vev+\vece_j-\vev}{\vey-\chi(\vey)+\chi(\vey)}{s+1-1}=\fibervert{\vex+\vece_j}{\vey}{s}.
\end{split}
\end{equation*}
Here, the second edge is feasible since $j\in\supp(\vew_1+s\cdot\mathbf{1}_k-\vex)$ by assumption \eqref{equ:StartEdgeWithinBox} and hence we have for the slack variable of $\vex+\vev$ that
$$j\in\supp\left(\vew_1+(s+1)\cdot\mathbf{1}_k-(\vex+\vev)\right).$$
All in all, we get in any case $2^k$ many edge-disjoint paths which only use edges
outside of $\cbox{s}{\veb}$ and hence these paths are \vertex-disjoint to those
walking within $\cbox{s}{\veb}$. Thus, there are
\begin{equation*}
|\supp(\vew_1+s\cdot\mathbf{1}_k)|+|\supp(\vew_2+(c-s)\cdot\mathbf{1}_k)|+2^k\ge\delta(\fibergraph{A_k}{\veb}{\Graver(A_k)})
\end{equation*}
\vertex-disjoint paths between the end-points of the edge given in \eqref{equ:StartEdgeWithinBox}.
\eoproof
\end{proof}
\par{In the next lemma we prove that we can find a suitable number of paths even for end-points of edges in neighbouring boxes of $\fiber{A_k}{\veb}$ as well. Here, Proposition~\ref{prop:CombiningGraphs} plays an important role and hence we shortly recall its statement: given two subgraphs with a certain connectivity yield a lower bound on the connectivity of the induced graph on the union of those subgraphs if we can prove the existence of a suitable number of paths walking between them.
In the situation of Proposition~\ref{lem:EdgesBetweenboxes}, the subgraphs whose connectivity is already known are the induced subgraphs on the boxes $\cbox{s}{\veb}$. So the idea behind the proof of Lemma~\ref{lem:EdgesBetweenboxes} is to find a sufficient number of edges between two neighbouring boxes.}

\begin{lemma}[Edges between adjacent Boxes]\label{lem:EdgesBetweenboxes}
Let $k>0$ and $\veb\in\Z^{2k+1}$. Then for any adjacent \vertices\ in
different boxes there are $\delta(\fibergraph{A_k}{\veb}{\Graver(A_k)})$ many edge-disjoint paths
connecting them.
\end{lemma}
\begin{proof}
By assumption, there exist at least two boxes in
$\fibergraph{A_k}{\veb}{\Graver(A_k)}$ and hence we must have $l(\veb)<u(\veb)$.
Without restricting generality, we can assume that the edge between the two adjacent \vertices\ looks like: 
\begin{equation}\label{equ:ReplaceEdge}
\veu_1:=\fibervert{\vex}{\vey}{s}\xlongleftrightarrow{\gr{\vev_1}{\vev_2}}\fibervert{\vex-\vev_1}{\vey+\vev_2}{s-1}:=\veu_2
\end{equation}
with $s>l(\veb)$ and $\vev_1,\vev_2\in\normalbox{\mathbf{1}_k}$. Let
us verify the assumptions of Proposition~\ref{prop:CombiningGraphs}. As already shown in Lemma~\ref{lem:PathsInbox}, the edge-connectivity in the two graphs $\normalgraph{\cbox{s-1}{\veb}}{\Graver(A_k)}$ and $\normalgraph{\cbox{s}{\veb}}{\Graver(A_k)}$ is at least
$$n:=\min_{j\in\{1,2\}}\{|\supp(\vew_j+\|\vew_j^-\|_\infty\cdot\mathbf{1}_k)|\}+k.$$
Since we have $m:=2^k\ge 2k-2\ge n-2$ it is left to prove that there are $2^k$ \vertex-disjoint paths connecting $\veu_1$ with $\veu_2$ and which only use edges between $\cbox{s-1}{\veb}$ and $\cbox{s}{\veb}$. For this, we define the sets 
\begin{equation}
\begin{split}
W_s&:=\{\fibervert{\vex-\vev_1+\vez}{\vey}{s}: \vez\in\normalbox{\mathbf{1}_k}\}\subset\cbox{s}{\veb}\\
W_{s-1}&:=\{\fibervert{\vex-\vev_1}{\vey+\vez}{s-1}: \vez\in\normalbox{\mathbf{1}_k}\}\subset\cbox{s-1}{\veb}.
\end{split}
\end{equation}
It is easy to see that $W_s$ is completely contained in the neighborhood of every \vertex\ in $W_{s-1}$ and vice versa. This means that $\fibergraph{A_k}{\veb}{\Graver(A_k)}$ has a complete bipartite graph on the \vertex\ sets $W_s$ and $W_{s-1}$ as subgraph including our original edge \eqref{equ:ReplaceEdge}. This gives $2^k$ many \vertex-disjoint paths between $\veu_1$ and $\veu_2$ only using edges between $\cbox{s}{\veb}$ and $\cbox{s-1}{\veb}$. Applying Proposition~\ref{prop:CombiningGraphs}, we obtain $m+n=\delta(\fibergraph{A_k}{\veb}{\Graver(A_k)})$ edge-disjoint connecting paths connecting $\veu_1$ and $\veu_2$.
\eoproof
\end{proof}
\par{Combining all the results of this section, we obtain our main theorem.}
\begin{theorem}\label{thm:GraverConnectivity}
For $k>0$, the edge-connectivity in all Graver fiber graphs of $A_k$ equals its minimal degree. 
\end{theorem}
\begin{proof}
From Lemma~\ref{lem:EdgeConnectivity} we know that we only have to consider paths between adjacent \vertices. From the decomposition of the fiber $\fiber{A}{\veb}$ given in \eqref{equ:FiberPresentation} we obtain that there are only two kinds of edges: edges within boxes and edges connecting two neighboring boxes. Lemma~\ref{lemma:ReplacingEdgesinbox} and Lemma~\ref{lem:EdgesBetweenboxes} state that we found in both cases $\delta(\fibergraph{A_k}{\veb}{\Graver(A_k)})$ many edge-disjoint paths connecting the adjacent \vertices\ of that edge.\eoproof{} 
\end{proof}

\par{Unfortunately, Theorem~\ref{thm:GraverConnectivity} says nothing about the \vertex-connectivity of the fiber graphs and we do not know whether it is best possible or not. Nevertheless, the results of this section make us suggest that requiring the Graver basis as set of edges should suffice that the edge-connectivity (not the \vertex-connectivity!) equals the minimal degree in all fiber graphs of arbitrary integer matrices.}

\begin{conjecture}Let $A\in\Z^{d\times n}$ be an integer matrix with $\ker(A)\cap\Zge^n=\{\mathbf{0}_n\}$. Then in all Graver fiber graphs of $A$, the edge-connectivity equals its minimal degree.
\end{conjecture}

\section{Computational Results}\label{sec:ComputationalResults}
\par{In this section, we present how random walks on fiber graphs of
$A_k$ behave. Therefore, let us first introduce briefly the
framework. Let $G=(\{v_1,\ldots,v_n\},E)$ be a simple graph. Consider the random walk which has
for $i,j\in[n]$ the probability
\begin{equation*}\label{equ:Metropolis}
	p_G(v_i,v_j)=
		\begin{cases}
\min\{1/\deg(v_i),1/\deg(v_j)\},&\textnormal{if }\{v_i,v_j\}\in	E\textnormal{ and }i\neq j\\
\sum_{\{v_i,v_k\}\in E}\max\{0,1/\deg(v_i)-1/\deg(v_k)\},&\textnormal{if }i=j\\
0,&\textnormal{if }\{v_i,v_j\}\not\in E\\
		\end{cases}
\end{equation*}
to traverse from $v_i$ to $v_j$. The matrix
$P_G=(p_G(v_i,v_j))_{i,j\in[n]}$ is precisely the transition
probability matrix of the Metropolis-Hastings chain on $G$ whose stationary distribution is the
uniform distribution on $\{v_1,\ldots,v_n\}$~\cite[Section~1.2.2]{Boyd2004}.
Given a vertex $v_i$ and a time step $t\in\N$, the
$j$th-entry of the vector $P_G^t\cdot\vece_i\in[0,1]^n$ is the probability that
a random walk starting at $v_i$ is at $v_j$ in time step $t$.
Let $\mu(P_G)\in[0,1]$ be the \emph{second largest eigenvalue modulus} (SLEM) 
of $P_G$. Since $(P_G^t\cdot\vece_i)_{t\in\N}$ converges to uniform
$\frac{1}{n}\cdot\mathbf{1}_n$ asymptotically with
$\mu(P_G)^t$~\cite[Section~1.1.2]{Boyd2004}, $\mu(P_G)$ is an
indicator of how fast the convergence of the
corresponding Markov chain towards its stationary distribution is.}
\par{In our experiments with \emph{Macaulay2}
\cite{Macaulay} we considered this random walk
on the fiber graphs
$\fibergraph{A_k}{\vece_{2k+1}}{\Graver(A_k)}$ and
$\fibergraph{A_k}{\vece_{2k+1}}{\GB}$, respectively.
\begin{figure}[htbp]
\begin{minipage}{0.5\textwidth}
\begin{tikzpicture}
	\begin{axis}[
		xlabel={$k$},
		ylabel={SLEM},
		legend style={at={(0.7,0.1)},anchor=south,draw=none},
		legend entries={Graver, Gr\"obner}
		]
		\addplot table {\fEslemgraver};
		\addplot table {\fEslemgroebner};
	\end{axis}
\end{tikzpicture}
\end{minipage}
\begin{minipage}{0.5\textwidth}
\begin{tikzpicture}
	\begin{axis}[
		xlabel={$k$},
		ylabel={Mixing Time},
		legend style={at={(0.3,0.3)},anchor=south,draw=none},
		legend entries={Graver, Gr\"obner}
		]
		\addplot table {\fEmixingtimegraver};
		\addplot table {\fEmixingtimegroebner};
	\end{axis}
\end{tikzpicture}
\end{minipage}
\caption{Plots of SLEM and mixing time of $\fiber{A_k}{\vece_{2k+1}}$ with respect to Graver and Gr\"obner moves.}\label{fig:PlotExp1}
\end{figure}
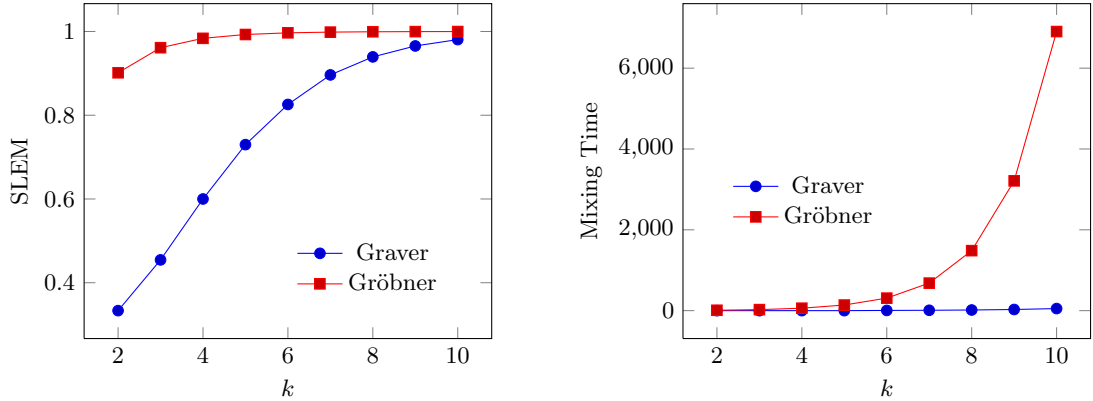
The left plot of Figure~\ref{fig:PlotExp1} shows how the SLEM of those
chains behaves if $k$ rises. It seems that both the SLEM of the
Gr\"obner chain and the SLEM of the Graver chain tend to $1$ as $k$
rises. The difference of the convergence of those two graphs becomes
even more visible by plotting their mixing times. Whereas the mixing time of $\fiber{A_{10}}{\vece_{10}}$ with Gr\"obner moves is around $7000$, the mixing time of the same fiber using Graver moves instead is approximately $50$.}
\begin{figure}[htbp]
\begin{minipage}{0.5\textwidth}
\begin{tikzpicture}
	\begin{axis}[
		xlabel={$\lambda$},
		ylabel={SLEM},
		legend style={at={(0.7,0.1)},anchor=south,draw=none},
		legend entries={Graver, Gr\"obner}
		]
		\addplot table {\sEslemgraver};
		\addplot table {\sEslemgroebner};
	\end{axis}
\end{tikzpicture}
\end{minipage}
\begin{minipage}{0.5\textwidth}
\begin{tikzpicture}
	\begin{axis}[
		xlabel={$\lambda$},
		ylabel={Mixing Time},
		legend style={at={(0.7,0.3)},anchor=south,draw=none},
		legend entries={Graver, Gr\"obner}
		]
		\addplot table {\sEmixingtimegraver};
		\addplot table {\sEmixingtimegroebner};
	\end{axis}
\end{tikzpicture}
\end{minipage}
\caption{Plots of SLEM and mixing time of $\fiber{A_3}{\lambda\cdot\vece_{7}}$ with respect to Graver and Gr\"obner moves.}\label{fig:PlotExp2}
\end{figure}
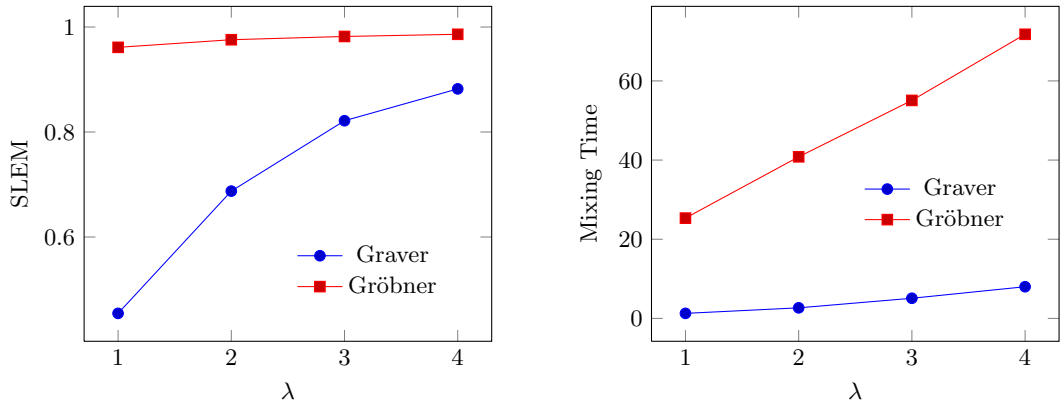
\par{In another experiment, we fixed $k=3$ and we computed SLEM and
mixing time of $\fiber{A_3}{\lambda\cdot\vece_7}$ with respect
to Graver and Gr\"obner moves for rising $\lambda\in\N$. Even if we do not know the connectivity of Gr\"obner fiber graphs of $A_3$ for right-hand sides $\veb\neq\vece_{7}$ in general, the Graver moves lead in our tested cases to a substantial better mixing time.}
\par{{\bf ACKNOWLEDGEMENTS.} The second author was supported by TopMath, a graduate program of the Elite Network of Bavaria and the TUM Graduate School. He further acknowledges support from the German National Academic Foundation.}

\bibliographystyle{plain}

\end{document}